\documentclass[11pt, reqno]{amsart}
\usepackage{amssymb, amstext, amscd, amsmath, amsthm}
\usepackage{color}
\usepackage{dsfont}
\usepackage{fullpage}
\usepackage{crossreftools} 
\usepackage{hyperref}
\usepackage{enumitem}

\usepackage{graphicx,tikz}
\usetikzlibrary{calc,decorations.markings}
\usetikzlibrary{arrows.meta}

\usepackage{xy}
\xyoption{all}

\theoremstyle{plain}
\newtheorem{theorem}{Theorem}[section]
\newtheorem{thm}[theorem]{Theorem}
\newtheorem{cor}[theorem]{Corollary}
\newtheorem{prop}[theorem]{Proposition}
\newtheorem{lem}[theorem]{Lemma}
\newtheorem*{theorem*}{Theorem}
\theoremstyle{definition}
\newtheorem{rem}[theorem]{Remark}
\newtheorem{defn}[theorem]{Definition}

\newtheorem{eg}[theorem]{Example}

\newcommand{\DirectedGraph}{E}

\newcommand{\bF}{{\mathbb{F}}}

\newcommand{\bN}{{\mathbb{N}}}
\newcommand{\bO}{{\mathbb{O}}}

\newcommand{\bT}{{\mathbb{T}}}

\newcommand{\bZ}{{\mathbb{Z}}}

  \newcommand{\A}{{\mathcal{A}}}

\renewcommand{\H}{{\mathcal{H}}}

  \newcommand{\K}{{\mathcal{K}}}
\renewcommand{\L}{{\mathcal{L}}}

\renewcommand{\O}{{\mathcal{O}}}

\renewcommand{\S}{{\mathcal{S}}}
  
  \newcommand{\U}{{\mathcal{U}}}
  \newcommand{\V}{{\mathcal{V}}}
  \newcommand{\W}{{\mathcal{W}}}

\newcommand{\upchi}{{\raise.35ex\hbox{\ensuremath{\chi}}}}

\newcommand{\qforal}{\quad\text{for all}\quad}

\newcommand{\Aut}{\operatorname{Aut}}

\newcommand{\ran}{\operatorname{Ran}}
\newcommand{\spn}{\operatorname{span}}

\newcommand{\BS}{\operatorname{BS}}

\begin{document}
\title[Wold Decomposition and C*-envelope of Self-Similar Graphs]{Wold Decomposition and C*-envelopes of Self-Similar
Semigroup Actions on Graphs}
\author[B. Li]{Boyu Li}
\address{Department of Mathematical Sciences, New Mexico State University, Las Cruces, New Mexico, 88003, USA}
\email{boyuli@nmsu.edu}
\author[D. Yang]{Dilian Yang}\thanks{}
\address{Dilian Yang,
Department of Mathematics $\&$ Statistics, University of Windsor, Windsor, ON
N9B 3P4, CANADA}
\email{dyang@uwindsor.ca}

\thanks{}
\thanks{The second author was partially supported by an NSERC Discovery Grant.}

\begin{abstract} 
We study the Wold decomposition for representations of a self-similar semigroup $P$ action on 
a directed graph $E$. We then apply this decomposition to the case where $P=\bN$ to study the C*-envelope of the associated universal non-selfadjoint operator algebra 
$\A_{\bN, E}$
by carefully constructing explicit non-trivial dilations for non-boundary representations. In particular, it is shown that the C*-envelope of $\A_{\bN, E}$ 
coincides with the self-similar graph C*-algebra $\O_{\bZ, E}$. 

\end{abstract}


\maketitle

\section{Introduction} 

Every non-selfadjoint operator algebra is embedded inside the smallest C*-algebra, known as its C*-envelope. Computing the C*-envelopes of various operator algebras is an active area of research \cite{DK2011, DM2017, MS1998} since Arveson first introduced this concept in \cite{ArvesonSubalgI}. In particular, many non-selfadjoint operator algebras arise from semigroup dynamics \cite{DFK2014, Humeniuk2020, BLi2021}. One major conjecture in the study of semigroup C*-algebra is the question of whether the C*-envelope of {a} non-selfadjoint semigroup operator algebra is the boundary quotient C*-algebra. Significant progress towards this problem was made by Sehnem recently \cite{Sehnem2022}, where she proved that the C*-envelope of the tensor algebra of a sub-semigroup of a group is the boundary quotient C*-algebra, using the coaction technique introduced in the recent work of Dor-on et al. \cite{DKKLL2022}. It is worth mentioning that the main result of \cite{Sehnem2022} has been generalized and simplified very recently in \cite{BBD2023}.
However, the C*-envelope of the universal non-selfadjoint semigroup operator algebra remains open. 

One way to compute the C*-envelope is through a dilation theory technique developed by Dritschel and McCollough in \cite{DM2005}. They proved that the C*-envelope {is} generated by maximal representations, which are representations without non-trivial dilations. Therefore, computing the C*-envelope is equivalent to characterising maximal representations. To prove the C*-envelope is generated by a certain type of representations, it suffices to prove the following: (1) such representations are maximal; and (2) all other representations have non-trivial dilations. This approach has been successfully adopted in the study of certain semigroup dynamics \cite{BLi2021, Wiart2016}. This strategy often hinges on a clear characterization of the representation. 

In the representation of a semigroup, each semigroup element is represented by an isometry. The characterization of isometric representations dates back to the work of von Neumann, where he showed an isometric operator on a Hilbert space can be decomposed as a direct sum of a unitary and a pure isometry (i.e. a multiple of the unilateral shift on $\ell^2(\mathbb{N})$). This is known as a Wold-decomposition, and it is an important result with many applications. To compute the universal non-selfadjoint operator algebra generated by a single isometry, it suffices to see that {a unitary is maximal} while {the unilateral shift has} non-trivial dilations. As a result, the C*-envelope is the universal C*-algebra generated by {the unitary}, which is precisely $C(\mathbb{T})$. Moving towards more complicated semigroups, the Wold decomposition is often unclear. For example, even for a simple semigroup as $\mathbb{N}^2$, whose representations are generated by a pair of commuting isometries, its Wold decomposition is tricky. The case when the two isometries doubly commute was first studied in \cite{Slocinski1980}, but the general case is only settled in \cite{Popovici2004} more than 20 years later. A Wold-type decomposition is only known in a handful of contexts, including doubly-commuting isometries \cite{Sarkar2014, Slocinski1980}, row isometries \cite{Popescu1998_Wold}, {Toeplitz-Cuntz-Krieger} families of directed graphs \cite{DDL2019, JK2005}, and product systems {\cite{SZ2008}}. 

Motivated by the recent work of the first-named author \cite{BLi2022} on the Wold-type decomposition for odometer semigroups, we seek to develop a Wold-type theorem for isometric Nica-covariant representations on {more} general semigroups. The Nica-covariance condition originated from Nica's seminal work on semigroup operator algebras \cite{Nica1992}. In the case of $\mathbb{N}^2$, the Nica-covariance condition is precisely the doubly commuting condition considered in \cite{Slocinski1980}. This paper considers the class of {semigroups} (and more generally, small categories) arising from {self-similar $\mathbb{N}$-actions on directed graphs}. It contains a number of interesting classes of semigroups, including the odometer semigroup and the Baumslag-Solitar {semigroups} \cite{BRRW}. Its representation consists of a single isometry for $\mathbb{N}$ and a Toeplitz-Cuntz-Krieger family for the directed graph. We prove a Wold-type decomposition for these {representations:
Such a representation decomposes into 4 components}, corresponding to the 4 combinations where the isometry for $\mathbb{N}$ is pure or unitary, and the {Toeplitz-Cuntz-Krieger} family for the directed graph is pure shift or {Cuntz-Krieger}. 

We then {apply} this Wold decomposition to compute the C*-envelope of {the universal non-selfadjoint operator algebra associated to a self-similar $\mathbb{N}$-action on a directed graph}. Among the four {combinations} in its Wold decomposition, we prove that only the unitary + CK component is maximal. As a result, the unitary + CK type representations generate its C*-envelope, which coincides with the self-similar graph C*-algebra defined in \cite{EP2017}. Compared with the recent works of {Sehnem \cite{Sehnem2022}, Dor-on et al. \cite{DKKLL2022}, and Brix et al. \cite{BBD2023}}, we adopted a {\textit{constructive}} approach where we {\textit{explicitly}} construct non-trivial {dilations} in proving certain types of representations are not maximal. In particular, the dilation technique we {used} for the unitary + pure shift type is highly non-trivial. As a result, we are able to compute the C*-envelope of the \emph{universal} non-selfadjoint operator algebra, instead of the \emph{reduced} algebra (i.e. the tensor algebra generated by the left regular representation) in the aforementioned papers.
We hope this work can be further generalized to other semigroups and shed some light on the C*-envelope problem of universal semigroup operator algebras using dilation theory techniques. We would like to thank Kenneth Davidson and Marcelo Laca for their valuable comments and suggestions.

\section{Self-similar graphs}

The main object of this paper comes from the operator algebra arising from self-similar semigroup actions on directed graphs. We first give a brief introduction to their algebraic structures and representations in this section. 

Group and semigroup actions play a vital role in the construction of various structures. For example, the semi-direct product of two groups encodes a one-way automorphic action of one group on another. More generally, two groups may have two-way interactions between them, resulting in a more complex structure that encodes this two-way interaction, also known as the Zappa-Sz\'{e}p product \cite{Szep1950, Zappa1940}. For the dynamics between a group and a directed graph, this two-way interaction is closely related to the notion of the self-similar action, originated from Grigorchuk's construction of finitely generated groups of intermediate growth \cite{Grigorchuk1984, Grigorchuk1980}. Similar notions can also be made for semigroup actions on directed graphs \cite{BGN2003, Nekrashevych2006}. In the field of operator algebras, the representations of such structures often lead to curious interactions of the underlying representations, and the study on these representations and their operator algebras has been very active recently. See, for example, the case of self-similar groups \cite{Nekrashevych2009}, self-similar graph \cite{EP2017}, self-similar higher rank graphs \cite{LY2019, LY-ETDS, LY2019-IMRN}, semigroups \cite{BRRW, Starling2015}, groupoids \cite{BPRRW2017}, Fell bundles \cite{DL2022}, small categories \cite{BKQS, XLiSmallCategory}, and product systems \cite{LY2022}.

In this paper, a self-similar graph encodes a similar action between a \emph{semigroup} and a directed graph. Most of our analysis focuses on the case when the semigroup is simply $\mathbb{N}$. Even in such a simple case, it already encompasses a wide range of interesting examples, including some Zappa-Sz\'ep products considered in \cite{BRRW} and the `standard' Baumslag-Solitar semigroups. In general, our construction leads to a small category. C*-algebras of some small categories 
have been studied recently in \cite{Spielberg2020}. 

We now give an overview of our construction of self-similar graphs. 
A directed graph $\DirectedGraph=(\DirectedGraph^0,\DirectedGraph^1,r,s)$ {consists of a set $\DirectedGraph^0$ of vertices, a set $\DirectedGraph^1$ of directed edges, and two mappings $r,s: \DirectedGraph^1\to\DirectedGraph^0$ specifying the range and the source of each edge, respectively. }
We often use $\DirectedGraph^d$ to denote directed paths in $\DirectedGraph$ with length $d$. We also use $|\mu|$ to denote the length of a path $\mu$. We use $\DirectedGraph^*:=\cup_{d\geq 0} \DirectedGraph^d$ to denote the set of all finite directed paths.
For $v\in \DirectedGraph^0$ and $n\in \bN$, $vE^n$ is the set of all paths of length $n$ with range $v$; and similarly for $E^n v$. We also use $vE^*$ and $E^*v$ to denote all finite paths with range and source $v$ respectively.
\emph{We always assume all graphs in this paper are directed, finite, and source-free.} 

We say {that} a bijective map $\pi:E^0\cup E^1 \to E^0\cup E^1$ is an automorphism of the graph $E$ if 
$\pi$ maps $E^i$ to $E^i$ ($i=0,1$) and it commutes with the source and range maps. 
We use $\Aut(\DirectedGraph)$ to denote the set of all automorphisms of $\DirectedGraph$. 

Recall that a semigroup is a set $P$ with an associative multiplication. For simplicity, we always assume a semigroup is embedded inside a group, contains the identity $1_P$, and is discrete and countable. We say that \textit{$P$ acts on $\DirectedGraph$} if there is a semigroup homomorphism $\varphi$ from $P$ to $\Aut(\DirectedGraph)$.
For $p\in P$ and $\mu\in\DirectedGraph$, we often simply write $\varphi_p(\mu)$ as $p\cdot \mu$.
Obviously, for every $p\in P$ and $\mu \in \DirectedGraph^d$ ($d\geq 0$), there is a unique $\nu\in\DirectedGraph^d$ such that $p\cdot \nu = \mu$. 
This simple fact will be used frequently later without any further mention. 

Following the definition of self-similar group actions on graphs in \cite{EP2017}, we first define self-similar semigroup actions on graphs. 

\begin{defn}
\label{D:ss}
An action of $P$ on $E$ is called \emph{self-similar} if there is a restriction map $P\times (E^0\cup E^1)\to P, \ (p, \mu)\mapsto p|_\mu$ such that 
\begin{enumerate}[label=\textup{(S\arabic*)}, start=1]
\item\label{cond:S1} $(pq)|_\mu = p|_{q\cdot \mu} q|_\mu$ for all $p,q\in P$ and $\mu \in E^0\cup E^1$;
\item\label{cond:S2} $p|_v=p$ for all $p\in P$ and $v\in E^0$;
\item\label{cond:S3} $1_P|_e=1_P$ for all $e\in E^1$;
\item\label{cond:S4} $p|_e\cdot v=p\cdot v$ whenever $s(e)=v$. 
\end{enumerate}
In this case, We call $\DirectedGraph$ a \textit{self-similar graph over $P$}, or $(P,\DirectedGraph)$ a \textit{self-similar graph}. 
\end{defn}
One should note that when $P$ is a group, this is precisely the definition of a self-similar graph defined in \cite{EP2017}. 

\begin{rem}
\label{R:ss}
Given a self-similar graph $(P,\DirectedGraph)$, one can naturally extend the ambient {action and restriction} to $E^*$ 
satisfying the following properties: 
\begin{enumerate}[label=\textup{(S\arabic*')}, start=1]
\item\label{item:S1}
$p\cdot (\mu\nu)=(p \cdot \mu)(p \vert_\mu \cdot \nu)$ for all $p \in P,\mu,\nu \in \DirectedGraph^*$ with $s(\mu)=r(\nu)$;

\item\label{item:S2}
$p \vert_v =p$ for all $p \in P,v \in \DirectedGraph^0$;

\item\label{item:S3}
$p \vert_{\mu\nu}=p \vert_\mu \vert_\nu$ for all $p \in P,\mu,\nu \in \DirectedGraph^*$ with $s(\mu)=r(\nu)$;

\item\label{item:S4}
$1_P \vert_{\mu}=1_P$ for all $\mu \in \DirectedGraph^*$;

\item\label{item:S5}
$(pq)\vert_\mu=p \vert_{q \cdot \mu} q \vert_\mu$ for all $p,q \in P,\mu \in \DirectedGraph^*$.
\end{enumerate}
\end{rem}

For each vertex $v\in \DirectedGraph^0$, we use $\Omega_v$ to denote the orbit of $v$ under the $P$-action, that is, $\Omega_v=\{p\cdot v: p\in P\}$.
Since we assume $\DirectedGraph$ is a finite graph and the $P$-action is bijective, we have that for each $q\in P$, $v=q^k \cdot v$ for some finite $k\geq 1$. Therefore, for each $q\in P$, $\Omega_{q\cdot v}=\Omega_v$. In other words, the relation $v\sim q\cdot v$ defines an equivalence relation on $\DirectedGraph^0$ and its equivalence classes are precisely $\Omega_v$.   
Analogously, for each $e\in \DirectedGraph^1$,  by $\Omega_e$ we denote the orbit of $e$ under the $P$-action. By the same argument, each $\Omega_e$ is an equivalence class of the equivalence relation $e\sim q\cdot e$. 

Every self-similar graph is naturally associated with a small category structure. Let $\DirectedGraph^* \bowtie P=\{(\mu, p): \mu\in E^*, p\in P\}$ and define the multiplication by
\[
(\mu,p)\cdot (\nu,q) := (\mu(p\cdot \nu), p|_\nu q) \text{ whenever } s(\mu)=r(p\cdot \nu).
\]
We call the small category $\DirectedGraph^*\bowtie P$ the \textit{self-similar product of $\DirectedGraph$ by $P$}. 
In particular, if $\DirectedGraph$ has only one vertex and $n$ edges, $\DirectedGraph^*\bowtie P$ becomes a semigroup, in which case we call it the \textit{self-similar product semigroup of $\DirectedGraph$ by $P$}. This semigroup is precisely the Zappa-Sz\'{e}p product of the free semigroup $\bF^+_n$ with $P$ \cite{BRRW}. 
This construction includes several important classes of semigroups. Here we demonstrate two classes of semigroups that are the main motivation behind this paper. One can refer to \cite[Section 3]{BRRW} for more examples. 

\begin{eg} The odometer semigroup $\bO_n$ is generated by $n$ free generators $\{e_1,\dots,e_n\}$ and an additional generator $v$ such that 
\[ve_i=\begin{cases}
e_{i+1}, & \text{ if } 1\leq i\leq n-1, \\
e_1 v, & \text{ if } i=n.
\end{cases}\]
This can be realized as a self-similar graph as follows. Let $\DirectedGraph$ be a graph with a single vertex and $n$ edges $\{e_1,\dots,e_n\}$. A self-similar action by $\bN$ on $E$ is defined by the $\bN$-action
\[
1\cdot e_i = \begin{cases}
e_{i+1}, & \text{ if } 1\leq i\leq n-1, \\
e_1, & \text{ if } i=n,
\end{cases}
\]
and the restriction map 
\[1|_{e_i} = \begin{cases}
0, & \text{ if } 1\leq i\leq n-1, \\
1, & \text{ if } i=n.
\end{cases}
\]
One can verify that these maps define a self-similar graph $(\bN,\DirectedGraph)$. Moreover, the self-similar product semigroup $\DirectedGraph^*\bowtie \bN$ is precisely the odometer semigroup $\bO_n$. 
\end{eg}

\begin{eg} The Baumslag-Solitar semigroups $\BS^+(n,m)$ $(1\le m, n\in \bN)$ are among the first few non-trivial examples of quasi-lattice ordered semigroups, and as a result, attract a lot of attention in the study of semigroup C*-algebras (\cite{ABCD2021, CaHR2016, Spielberg2012}). For two positive integers $n,m\ge 1$, $\BS^+(n,m)$ is a semigroup generated by two generators $a,b$ with the relation $a^n b=ba^m$. Notice that the elements $b, ab, \dots, a^{n-1}b$ are free inside $\BS^+(n,m)$, so they generate a copy of the free semigroup $\bF_n^+$ inside $\BS^+(n,m)$. 

Let $e_i=a^{i-1}b$ for $1\leq i\leq n$. The $\BS^+(n,m)$ can be seen as a self-similar product semigroup where $P=\bN=\langle a\rangle$ act on $\bF^+_n=\langle e_1,\ldots,e_n\rangle$ via an $\bN$-action:
\[
a\cdot e_i = \begin{cases}
e_{i+1}, & \text{ if } 1\leq i\leq n-1, \\
e_1, & \text{ if } i=n,
\end{cases}
\]
{and} $\bN$-restriction:
\[a|_{e_i} = \begin{cases}
0, & \text{ if } 1\leq i\leq n-1, \\
a^m, & \text{ if } i=n.
\end{cases}
\]
Clearly, when $m=1$, this is precisely the self-similar action for the odometer semigroup $\bO_n$. So $\BS^+(n,1)$ coincides with $\bO_n$. 
\end{eg}

To study the operator algebras associated with self-similar actions, it is useful to encode the self-similar dynamics by representations. Representations of semigroups and graphs have been well-studied in the literature respectively. Here, we first briefly review their representations before combining them to encode self-similar actions. 

\begin{defn} Let $\DirectedGraph$ be a graph. A \emph{{Toeplitz-Cuntz-Krieger} (TCK) $\DirectedGraph$-family} is given by 
$\S:=\{S_\mu: \mu\in \DirectedGraph^0\cup \DirectedGraph^1\}$ such that
\begin{enumerate}[label=\textup{(TCK\arabic*)}, start=1]
    \item\label{item:TCK1} For each $v\in \DirectedGraph^0$, $S_v$ is an orthogonal projection. Moreover, $\{S_v\}_{v\in\DirectedGraph^0}$ is a pairwise orthogonal family.
    \item\label{item:TCK2} For each $e\in \DirectedGraph^1$, $S_e^* S_e=S_{s(e)}$.
    \item\label{item:TCK3} For each $v\in \DirectedGraph^0$, 
    \[\sum_{e\in v\DirectedGraph^1} S_e S_e^* \leq S_v.\]
\end{enumerate}
We say this family is a \emph{{Cuntz-Krieger} (CK) family} if \ref{item:TCK3} is replaced by 
\begin{equation} \tag{CK}
\sum_{e\in v\DirectedGraph^1} S_e S_e^* = S_v \text{ for all }v\in \DirectedGraph^0.
\end{equation}
\end{defn}

It is standard in the literature that TCK  and CK $\DirectedGraph$-families can be extended to $\DirectedGraph^*$ \cite{Raeburn2005}.  

\begin{eg}\label{ex.LEv} It is worth mentioning a special class of TCK $\DirectedGraph$-families -- the pure shifts. 
For this, let us recall the Fock representation of $\DirectedGraph$.
For $\mu \in \DirectedGraph^0\cup \DirectedGraph^1$, define the operator $L_\mu:\ell^2(\DirectedGraph^*)\to \ell^2(\DirectedGraph^*)$ as follows:
\begin{align*}
L_\mu(\delta_\nu):=\begin{cases}
   \delta_{\mu \nu}   &\text{ if $s(\mu)=r(\nu)$} \\
   0   &\text{ otherwise}.
\end{cases}
\end{align*}
One can check that $\L=\{L_\mu: \mu \in  \DirectedGraph^0\cup \DirectedGraph^1\}$ is a TCK $\DirectedGraph$-family. For each $v\in \DirectedGraph^0$, denote $\H_v=\overline{\spn}\{\delta_\nu: \nu\in\DirectedGraph^*, s(\nu)=v\}$. Since each $L_\mu$ append $\mu$ while $L_\mu^*$ cancel $\mu$ on the left only, each $\H_v$ is a reducing subspace for $\L$. We use $\L_{\DirectedGraph,v}$ to denote the restriction of $\L$ on $\H_v$. A representation is called a \emph{pure shift $\DirectedGraph$-family} if it is a direct sum of some amplifications of $L_{\DirectedGraph,v}$'s. Such families are also called the left regular type in the literature. 
\end{eg}

\begin{defn} Let $P$ be a semigroup. An isometric representation of $P$ is given by a family of isometries $\V=\{V_p\}_{p\in P}$, such that $V_pV_q=V_{pq}$ for all $p,q\in P$. If $1_P$ is the identity of $P$, then we always assume $V_{1_P}=I$. 
\end{defn}

\begin{rem} Recent studies on semigroup C*-algebras often put more structures (e.g. quasi-lattice, LCM, constructible ideals) on the ambient semigroups. While it might be useful to require these additional structures, the focus of this paper is on the interplay between semigroups and directed graphs via self-similar actions. We shall leave the analysis on these additional structures for later works.  
\end{rem}

We now introduce representations associated with a self-similar semigroup action on a graph.

\begin{defn}
\label{D:repn}
Let $(P, \DirectedGraph)$ be a self-similar graph. A \emph{(Toeplitz) representation} of $(P, \DirectedGraph)$ consists of an isometric representation $\V$ of $P$ on a Hilbert space $\H$ and a TCK family $\S$ for $\DirectedGraph$ on $\H$, such that
\begin{enumerate}
    \item[{\crtcrossreflabel{(SS)}[E:CD]}] 
    $V_p S_\mu=S_{p\cdot \mu} V_{p|_\mu}$ for all $p\in P$ and $\mu\in\DirectedGraph^0\cup \DirectedGraph^1$, 
    \item[{\crtcrossreflabel{(NC)}[E:NC]}] 
    $V_p^* S_{p\cdot \mu}=S_\mu V_{p|_\mu}^*$ for all $p\in P$ and $\mu\in\DirectedGraph^0\cup \DirectedGraph^1$. 
\end{enumerate}
\end{defn}

We first note that Conditions \ref{E:CD} and \ref{E:NC} imply they hold true for all $\mu \in E^*$, respectively. 
In fact, for $\mu=e_1\dots e_k$ and $p\in P$, we can check that
\[V_p S_\mu = V_p S_{e_1}\cdots S_{e_k} = S_{p\cdot e_1} S_{p|_{e_1} \cdot e_2} \cdots S_{p|_{e_1\dots e_{k-1}}\cdot e_k} V_{p|_\mu} = S_{p\cdot \mu} V_{p|_\mu}.\]
A similar computation can be made for Condition \ref{E:NC}. 

\begin{rem} The condition \ref{E:NC} is related to the Nica-covariance condition arising from the representations of quasi-lattice ordered semigroups. Without this condition, the Wold decomposition can be very difficult to characterize, as we shall illustrate in the next example. 
\end{rem}

\begin{eg} A representation of the odometer semigroup $\bO_n$ is determined by a single isometry $V$ and a row isometry $(S_1,\dots,S_n)$. Here, the single isometry $V$ corresponds to the isometric representation $\V$ of $\mathbb{N}$ by setting $V_p=V^p$. The condition \ref{E:CD} ensures that $V$ and $S$ intertwine with respect to the self-similar action. In the special case of $n=1$, this is simply requiring $V$ and $S_1$ are two commuting isometries. 

Without the condition \ref{E:NC}, the Wold decomposition of such isometries can be difficult, even in the case when $n=1$, which corresponds to a pair of commuting isometries. Fully characterizing the Wold decomposition of a pair of commuting isometries is difficult, which has only been done recently by Popovici \cite{Popovici2004} (see also \cite{BLi2022} for a similar characterization on odometer semigroups). 

On the other hand, the Wold decomposition under the condition \ref{E:NC} is much nicer. It is shown in \cite{BLi2022} (see also \cite{Slocinski1980} for commuting isometry case) that such a representation decomposes into 4 components, each corresponding to whether $V$ is pure or unitary and whether $S$ is pure or row unitary. 

For readers familiar with the semigroup C*-algebras, the condition \ref{E:NC} is the same as the Nica-covariant condition if we treat the odometer semigroup as a quasi-lattice ordered semigroup. While Nica-covariant representations are not the focus of this current paper, we do suspect a more general Wold decomposition for Nica-covariant representations and we will leave it for future investigations. 
\end{eg}

We would like to point out that Condition \ref{E:NC} follows automatically in some special cases. 

\begin{prop}\label{prop.automatic.NC} Let $\V$ be an isometric representation of $P$ and $\S$ a TCK family for $\DirectedGraph$, such that $(\V,\S)$ satisfies the condition \ref{E:CD}. If either $\V$ is a unitary representation or $\S$ is a CK family, we have that $(\V,\S)$ satisfies the condition \ref{E:NC}. 
\end{prop}

\begin{proof} First suppose that $\V$ is unitary, we have for any $p\in P$ and $\mu\in\DirectedGraph^*$, 
\[
V_p^*S_{p\cdot \mu}=V_p^*S_{p\cdot \mu}V_{p|_\mu} V_{p|_\mu}^*=V_p^* V_p S_\mu V_{p|_\mu}^*=S_\mu V_{p|_\mu}^*.
\]

Now assume that $\S$ is CK, then $\sum_{\nu\in\DirectedGraph^m} S_\nu S_\nu^*=I$ for all $m\geq 0$. We have   
\begin{align*}
V_p^*S_{p\cdot \mu}
&=\sum_{\nu\in \DirectedGraph^{|\mu|}} S_\nu S_\nu^*\, V_p^*S_{p\cdot \mu}=\sum_{\nu\in \DirectedGraph^{|\mu|}} S_\nu V_{p|_\nu}^*  S_{p\cdot \nu}^* S_{p\cdot\mu}
  =S_\mu V_{p|_\mu}^* S_{p\cdot s(\mu)}\\
&=S_\mu (S_{p\cdot s(\mu)}V_{p|_\mu})^* = S_\mu (V_p S_{s(\mu)})^* = S_\mu  S_{s(\mu)} V_p^* =S_\mu V_p^*.
\end{align*}
Here, we used the fact that the $P$-action is a permutation of all paths of length $|\mu|$, so that $S_{p\cdot \nu}^* S_{p\cdot \mu}=0$ unless $\mu=\nu$. 
Moreover, $S_{p\cdot s(\mu)}$ is an orthogonal projection so that it is self-adjoint. 
\end{proof}

The following technical lemma is useful in later computations. 

\begin{lem}\label{lm.technical}
For each $e\in\DirectedGraph^1$ and $p\in P$, 
    \[V_p S_eS_e^* = S_{p\cdot e} S_{p\cdot e}^* V_p.\]
    Let $e=p\cdot f$ for some $f\in \DirectedGraph^1$. 
    Then
    \[V_p^* S_e S_e^* = S_fS_f^* V_p^*.\]
\end{lem}

We omit the proof here since it's an easy application of Conditions~\ref{E:CD} and \ref{E:NC}. 

We now introduce an archetypal example of a Toeplitz representation. This is often known as the left regular representation. 

\begin{eg}  
Let $(P, E)$ be a self-similar graph. 
We first consider a Fock-type representation of the small category $\DirectedGraph^*\bowtie P$. Consider $\K=\ell^2(\DirectedGraph^*\times P)$ where we use $\{\delta_{\mu,p}: \mu\in\DirectedGraph^*, p\in P\}$ to denote its orthonormal basis. 
For each $q\in P$ and $\nu\in\DirectedGraph^i$ ($i=0,1$), define 
\begin{align*}
V_q \delta_{\mu,p} &= \delta_{q\cdot \mu, q|_\mu p},\\
S_\nu \delta_{\mu,p} &= \begin{cases}
\delta_{\nu\mu, p}, &\text{ if } s(\nu)=r(\mu),\\
0, &\text{ otherwise.}
\end{cases}
\end{align*}
We claim that $(\S,\V)$ defines a representation of $(P, E)$. 

We first verify that $\V$ is an isometric representation of $P$. For any $(\nu_1,p_1)\neq (\nu_2,p_2)$, we either have $\nu_1\neq \nu_2$ so that $q\cdot \nu_1\neq q\cdot \nu_2$, or have $\nu_1=\nu_2$ and $p_1\neq p_2$ so $q|_{\nu_1} p_1\neq q|_{\nu_2} p_2$. This proves that each $V_q$ maps the orthonormal basis to an orthonormal set and thus is an isometry. For the identity $1_P\in P$, $V_{1_P}$ is clearly the identity map by Condition \ref{item:S4}. For any $q,r\in P$, and $(\mu,p)\in\DirectedGraph^*\times P$, 
\[V_qV_r\delta_{\mu,p} = \delta_{q\cdot (r\cdot \mu), q|_{r\cdot \mu} r|_\mu p} = \delta_{(qr)\cdot \mu, (qr)|_\mu p} = V_{qr} \delta_{\mu,p}.\]

The family $\S$ is clearly a TCK $\DirectedGraph$-family since for each fixed $p$, the family act as a pure shift on $\overline{\spn}\{\delta_{\mu,p}: \mu\in\DirectedGraph^*\}$. To see $\S$ is a Toeplitz representation. We first check Condition \ref{E:CD}. Take any $q\in P$ and $\nu\in\DirectedGraph^*$, and any basic vector $\delta_{\mu,p}$. We have
\begin{align*}
    V_q S_\nu \delta_{\mu,p} &= V_q \delta_{\nu\mu,p} \\
    &= \delta_{q\cdot (\nu\mu), q|_{\nu\mu} p} \\
    &= \delta_{(q\cdot \nu)(q|_\nu\cdot \mu), (q|_\nu)|_\mu p} \\
    &= S_{q\cdot \nu} \delta{q|_\nu\cdot \mu, (q|_\nu)|_\mu p} \\
    &= S_{q\cdot \nu} V_{q|_\nu} \delta_{\mu,p}.
\end{align*}
To prove Condition \ref{E:NC}, it suffices to proves it adjoint: For any $q\in P$ and $\nu\in\DirectedGraph^*$, we have $S_{q\cdot \nu}^* V_q = V_{q|_\nu} S_\nu^*$. Fix a basic vector $\delta_{\mu,p}$, we first observe that
\[
S_\nu^* \delta_{\mu,p} = \begin{cases}
\delta_{\mu_1,p}, & \text{ if }\mu=\nu\mu_1 \text{ for some $\mu_1\in E^*$}, 
\\
0, & \text{ otherwise}. 
\end{cases}
\]
Therefore, in the case when $\mu=\nu\mu_1$, we have
\begin{align*}
V_{q|_\nu} S_\nu^* \delta_{\mu,p} &= V_{q|_\nu} \delta_{\mu_1, p} \\
&= \delta_{q|_\nu \cdot \mu_1, (q|_{\nu})|_{\mu_1} p}\\
&= S_{q\cdot \nu}^* \delta_{(q\cdot \nu)(q|_\nu \cdot \mu_1, q|_{\nu\mu_1} p} \\
&= S_{q\cdot \nu}^* \delta_{q\cdot \mu, q|_\mu p} \\
&= S_{q\cdot \nu}^* V_q \delta_{\mu,p}.
\end{align*}
In the case when $\mu$ does not start with $\nu$, we have that $q\cdot \mu$ does not start with $q\cdot \nu$ as well. Therefore,
\[S_{q\cdot \nu} V_q \delta_{\mu,p} = S_{q\cdot \nu}^* \delta_{q\cdot \mu, q|_\mu p}=0.\]
Hence, we verified that $(\V,\S)$ defines a Toeplitz representation of $(P,\DirectedGraph)$. 

When $\DirectedGraph^0$ contains more than one vertex, this representation is reducible. For each $v\in\DirectedGraph^0$, let 
\[\K_v=\overline{\spn}\{S_\mu V_p \delta_{v,e}: \mu\in\DirectedGraph^*, p\in P\}=\overline{\spn}\{\delta_{\mu,p}: \mu\in\DirectedGraph^*, p\in P, s(\mu)=p\cdot v\}.\]
One can verify that $\K_v$ reduces the Fock representation $(\V,\S)$. We shall use the notation $(\lambda^P_v, \lambda^\DirectedGraph_v)$ to denote the restriction of $(\V,\S)$ on $\K_v$.

A representation of $(P, E)$ is called \textit{left regular} if it is a direct sum of amplifications of these $\lambda$'s.
\end{eg}

Left regular representations are closely related to the notion of ``wandering vectors". 

\begin{defn} Let $(\V,\S)$ be a Toeplitz representation of a self-similar graph $(P, E)$ on a Hilbert space $\H$. We say a unit vector $\xi\in\H$ is a \textit{wandering vector} of $(\V,\S)$ if the vectors $\{S_\mu V_p \xi: \mu\in\DirectedGraph^*, p\in P\}$ are pairwise orthogonal to each other, and each vector is either $0$ or a unit vector. 
\end{defn}

Now suppose $\xi$ is a wandering vector for $(\V,\S)$. Then let $\H_\xi=\overline{\spn}\{S_\mu V_p \xi: \mu\in\DirectedGraph^*, p\in P\}$. We have $\H_\xi$ is invariant for $(\V,\S)$. From the definition of wandering vectors, it is easy to see that the restriction of $(\V,\S)$ on $\H_\xi$ is unitarily equivalent to $(\lambda_v^P, \lambda_v^\DirectedGraph)$.

\section{Wold decomposition on self-similar graphs}

This section aims to provide a Wold decomposition for Toeplitz representations of self-similar graphs $(\bN, \DirectedGraph)$ (Theorem \ref{thm.Wold.main}). It 
is shown that every Toeplitz representation consists of four components: unitary + CK, pure + CK, unitary + pure shift, and left regular. 

\subsection{Wold decomposition for Toeplitz representations of self-similar graphs} 

The Wold decomposition for a TCK family of graphs dates back to \cite[Section 2]{JK2005} (see also \cite[Section 4]{DDL2019}). 

\begin{thm}[\cite{DDL2019, JK2005}]\label{thm.Wold.TCK}
Let $\DirectedGraph$ be a graph and $\S$ be a TCK $\DirectedGraph$-family on $\H$. For each $v\in \DirectedGraph^0$, define 
\[\W_v=\left(S_v-\sum_{e\in vE^1} S_eS_e^*\right)\H.\] 
Then each $\W_v$ is a wandering subspace for $\S$. Let 
\[\H_v=\bigoplus_{\mu\in\DirectedGraph^*} S_\mu \W_v.\]
Then $\H_v$ reduces $\S$ and $\S|_{\H_v}\cong L_{E,v}^{(\alpha_v)}$, where $\alpha_v=\dim \W_v$.

The space $\H$ decomposes as $\H=\H_C \oplus \left(\bigoplus_{v\in \DirectedGraph^0} \H_v\right)$ where 
\[\H_C=\bigcap_{d\ge 1}\bigoplus_{\mu\in E^d} S_\mu \H.\]
Here, $\H_C$ reduces $\S$ and $\S|_{\H_C}$ is a CK family.
\end{thm}

We first observe that the representation $\V$ behaves well with respect to the decomposition of $\H$ for $\S$.

\begin{prop}\label{prop.TCK.reduceV} Let $(\V,\S)$ be a Toeplitz representation of a self-similar graph $(P,\DirectedGraph)$ on $\H$. Let $\W_v$, $\H_v$, and $\H_C$ as in Theorem \ref{thm.Wold.TCK}. 
\begin{enumerate}
    \item For each $v\in \DirectedGraph^0$, set $\W_{[v]}=\bigoplus_{u\in\Omega_v} \W_u$. Then $\W_{[v]}$ reduces $\V$.
    \item For each $m\geq 0$, set $\H_{[v]}^m  = \bigoplus_{\mu\in\DirectedGraph^m} S_\mu \W_{[v]}$. Then $\H_{[v]}^m$ reduces $\V$. 
    \item For each $v\in \DirectedGraph^0$, set $\H_{[v]}=\bigoplus_{u\in\Omega_v} \H_u$. Then $\H_{[v]}$ reduces $\V$.
    \item $\H_C$ reduces $\V$ as well. 
\end{enumerate}
\end{prop}

\begin{proof} To prove $\W_{[v]}$ reduces $\V$, we first observe that $h\in \W_v$ if and only if 
$h=(S_v-\sum_{e\in v E^1} S_eS_e^*)h$. 
Take $h\in \W_v$, using computations from Lemma \ref{lm.technical}, we have
\begin{align*}
     V_p h 
   &= \left(S_{p\cdot v} - \sum_{e\in vE^1} S_{p\cdot e} S_{p\cdot e}^*\right) V_p h 
     = \left(S_{p\cdot v} - \sum_{f\in vE^1} S_{f} S_{f}^*\right) V_p h, \\
    V_p^* h 
  &= \left(S_{u} - \sum_{p\cdot f\in vE^1} S_{f} S_{f}^*\right) V_p^* h 
    = \left(S_{u} - \sum_{f\in uE^1} S_{f} S_{f}^*\right) V_p^* h, \quad \text{where }p\cdot u=v.
\end{align*}
Here, we used the fact that $p|_u=p$ for all $u\in\DirectedGraph^0$. This proves that $V_p \W_v = \W_{p\cdot v}$ and $V_p^* \W_v=\W_u$ when $p\cdot u=v$. In either case, $p\cdot v$ and $u$ are still in $\Omega_v$ and thus $\W_{[v]}$ reduces each $V_p$. 

Now for a fixed $m\geq 0$, for any $\mu\in\DirectedGraph^m$ and $h\in\W_{[v]}$, we have 
\[V_p S_\mu h =S_{p\cdot \mu} V_{p|_\mu} h \in S_{p\cdot \mu} \W_{[v]}.\] 
Here, $V_{p|_\mu} h \in \W_{[v]}$ since $\W_{[v]}$ reduces $\V$. Now $p\cdot \mu$ is also a path of length $m$, and thus $V_p S_\mu h \in \H_{[v]}^m$. Similarly, if $\mu=p\cdot \nu$ for some path $\nu\in\DirectedGraph^*$, we must have $\nu\in\DirectedGraph^m$ as well. 
Now, we have that, 
\[V_p^* S_\mu h = S_\nu V_{p|_\nu}^* h \in \H_{[v]}^m.\]
Therefore, $\H_{[v]}^m$ reduces $\V$ as well. 

Now since 
\[\H_{[v]}=\bigoplus_{\mu\in\DirectedGraph^*} S_\mu\W_{[v]} = \bigoplus_{m\geq 0} \bigoplus_{\mu\in\DirectedGraph^m} S_\mu \W_{[v]} = \bigoplus_{m\geq 0} \H_{[v]}^m.\]
We have $\H_{[v]}$ reduces $\V$. Finally, $\H_C$ is the orthogonal complement of the direct sum of all $\H_{[v]}$ and thus $\H_C$ reduces $\V$ as well. 
\end{proof}



\begin{cor}
\label{C:P=G}
Suppose that $(P, \DirectedGraph)$ is a self-similar graph. Let $(\V,\S)$ be a Toeplitz representation of $(P, \DirectedGraph)$ on a Hilbert space $\H$. 
Then $\H$ can be uniquely decomposed as follows 
\[
\H=\H_C \oplus \left(\bigoplus_{\Omega_v \in \DirectedGraph^0/\sim} \H_{[v]}\right)
\]
such that each component reduces $(\V,\S)$ and ${\S}|_{\H_C}$ is a CK family.
Furthermore, if $P$ is a group, then ${\S}|_{\H_{[v]}} \cong L_{\DirectedGraph,v}^{(\alpha_v)}$ with $\alpha_v=\dim \W_u$ for all $u\in\Omega_v$.  
\end{cor}

\begin{proof}
By Proposition \ref{prop.TCK.reduceV}, it suffices to prove that $\dim \W_u$ is constant among all $u\in\Omega_v$ when $P$ is a group. Notice that 
we have shown that for each $g\in P$, $V_g$ maps $\W_v$ to $\W_{g\cdot v}$. Since $P$ is a group, this can be reversed using $V_{g^{-1}}$ and thus $\dim \W_v=\dim \W_{g\cdot v}$ for all $g\in P$. 
\end{proof}

We want to turn our attention to further decomposing the isometric representation $\V$ on each component. The Wold decomposition for a general isometric semigroup representation is not well understood (see \cite{Popescu1998_Wold, Popovici2004, Suciu1968}, where a representation of semigroups often has a component other than unitary and left-regular). Therefore, we focus on the case when $P=\mathbb{N}$, on which the classical Wold decomposition applies. This is also motivated by the earlier research on the Wold decomposition of odometer semigroups \cite{BLi2022}.

\subsection{Case when \texorpdfstring{$P=\mathbb{N}$}{P=N}} An isometric representation $\V$ of $\bN$ is uniquely determined by a single isometry $V_1$ since $V_m=V_1^m$. For simplicity, we often use $V$ to denote the isometry $V_1$. 

Suppose $(\V,\S)$ is a Toeplitz representation of a self-similar graph $(\bN, \DirectedGraph)$ on a Hilbert space $\H$. Proposition \ref{prop.TCK.reduceV} shows that we can decompose $\H$ into a direct sum of $\H_C$ and $\H_{[v]}$'s, where each subspace reduces $\V$.
The goal of this section is to further decompose each of these components to reducing subspaces of $(\V,\S)$, based on whether $\V$ is unitary or pure. 

For some technical reasons, we need the following assumption:
\begin{enumerate}[label=\textup{(A)}]
\item\label{item:A1}
For every $e\in E^1$, there is $m\in \bN$ such that $m|_e\ne 0$. 
\end{enumerate}

Notice that the assumption \ref{item:A1} is equivalent to $n|_e\to \infty$ as $n\to \infty$.


\begin{rem} We would like to point out that the assumption \ref{item:A1} 
is equivalent to saying that there are sufficiently many $e$ with $1|_e\neq 0$. In fact, from Condition \ref{item:S5}, we have 
$m|_e = \sum_{i=1}^{m-1} 1|_{i\cdot e}$.
Therefore, \ref{item:A1} is also equivalent to the following: For any $e\in\DirectedGraph^1$, there exists $f\in\Omega_e$ such that $1|_f\neq 0$.
\end{rem}

\begin{rem}
\label{R:1|e m}
For $v\in \DirectedGraph^0$, it follows from Condition~\ref{cond:S4} that $1|_e \cdot v= 1 \cdot v$ whenever $e\in \DirectedGraph^1 v$. 
Thus we must have $1|e= 1 \!\! \mod |\Omega_v|$. In particular, if $|\Omega_v|>1$ for every $v\in \DirectedGraph^0$, then our assumption (A) is redundant. 
\end{rem}

From the remarks above, it is not surprising to see that many examples satisfy the assumption \ref{item:A1}. 
For example, in the case of the odometer semigroup $\bO_n$, the graph $\DirectedGraph$ is a single vertex with $n$ loops. The orbit of each edge $e$ consists of all the edges. Since $1|_{e_n}=1\neq 0$,  the assumption\ref{item:A1} is clearly satisfied. 

\medskip
We first consider the Wold decomposition of $\V$. We can write $\H=\H^U\oplus \H^S$, where $\V$ is a unitary on $\H^U$ and a pure shift on $\H^S$, respectively. Here, 
\begin{align*}
\H^U = \bigcap_{n\in\bN} V^n \H  \ \text{ and } \
\H^S =\bigoplus_{n\in \bN} V^n\left(\ker V^*\right).
\end{align*}
We first claim that both $\H^U$ and $\H^S$ reduce $\S$. 

\begin{prop} Both $\H^U$ and $\H^S$ reduce $\S$. 
\end{prop} 

\begin{proof} Since $\H=\H^U\oplus \H^S$, it suffices to prove both $\H^U$ and $\H^S$ are invariant for $\S$. 

For $\H^U$, pick $h\in\H^U$. By definition, for each $n\in\bN$, there exists $h_n\in \H$ such that $h=V^n h_n$. 

We first observe that for any $u\in\DirectedGraph^0$, $n|_u=n$ by Condition \ref{item:S2}. Since $\bN$-action is bijective, pick $w\in\DirectedGraph^0$ such that $n\cdot w=u$, and we have $V^n S_w=S_u V^n$. Therefore, 
\[S_u h=S_u V^n h_n=V^n S_w h_n \in V^n \H.\]

Next, we fix $e\in\DirectedGraph^1$. By Assumption \ref{item:A1}, $m|_e\to\infty$ as $m\to\infty$. Therefore, for each $n\geq 1$, there exists some $n'\geq n$ and $m\in\bN$ such that $m|_{e}=n'$. Let $f\in\DirectedGraph^1$ be the unique edge such that $m\cdot f=e$. Then we have,
\[S_e h = S_e V^{n'} h_{n'} = V^m S_f h_{n'} \in V^m \H.\]
Since $m$ can be arbitrarily large, we have $S_e h$ in the range for all $V^m$ and thus $S_e h \in \H^U$. 

For $\H^S$, we first observe that for a vector $h\in\H$, $h$ belongs to a summand of $\H^S$ if and only if there exists $n\geq 1$ such that $V^{*n}h=0$, in which case $h\in V^{n-1}\left(\ker V^*\right)$.
 Now fix $h\in\H^S$ and suppose $V^{*n}h=0$. For each $u\in \DirectedGraph^0$, we can find $w\in\DirectedGraph^0$ such that $u=n\cdot w$. Therefore, by Condition \ref{E:NC}, 
\[V^{*n} S_u h = S_w V^{*n} h =0.\]

For each $e\in\DirectedGraph^1$, we use Assumption \ref{item:A1} again to pick some $m\in\bN$ such that $m|_e=n' > n$. Let $f\in\DirectedGraph^1$ be the unique edge such that $m\cdot f=e$. Then We have
\[V^{*m} S_e h = S_f V^{*n'}h = 0.\]
Therefore, $\H^S$ is invariant under $\S$. 
\end{proof}

We now further decompose each component in $\H$ obtained in Corollary \ref{C:P=G} as follows. 
For each component $\H_i$ ($i=C$ or $[v]$), let $\H_i^U:=\H^U \cap \H_i$ and $\H_i^S:=\H^S \cap \H_i$. 
This decomposes each $\H_i$ as $\H_i^U\oplus \H_i^S$, on which $\V$ is unitary or pure. Since $\H^U$ and $\H^S$ reduce $\S$, each $\H_i^U$ and $\H_i^S$ reduce $\S$ as well for all $i$. This leads to the following Wold decomposition for Toeplitz representations. 

\begin{thm}\label{thm.Wold.main} Let $(\V,\S)$ be a Toeplitz representation of a self-similar graph $(\bN, \DirectedGraph)$ on a Hilbert space $\H$. Then $\H$ decomposes into a direct sum of reducing subspaces as follows:
\[\H=\H_C^U \oplus \H_C^S \oplus \left(\bigoplus_{\Omega_v \in \DirectedGraph^0/\sim} \H_{[v]}^U \right) \oplus \left(\bigoplus_{v\in\DirectedGraph^0} \H_v^S \right).\]
When we restrict $(\V,\S)$ on these reducing subspaces, we have the following components:
\begin{enumerate}[label=(\roman*)]
    \item\label{item.Wold.main.1} \textbf{Unitary + CK}: On $\H_C^U$, $\V$ is a unitary and $\S$ is a CK $\DirectedGraph$-family.
    \item\label{item.Wold.main.2} \textbf{Pure + CK}: On $\H_C^S$, $\V$ is pure and $\S$ is a CK $\DirectedGraph$-family.
    \item\label{item.Wold.main.3} \textbf{Unitary + pure shift}:  For each equivalence class $\Omega_v$, on $\H_{[v]}^U$, $\V$ is unitary and $\S$ is a direct sum of some amplifications of $\bigoplus_{u\in\Omega_v} L_{\DirectedGraph,u}$.
    \item\label{item.Wold.main.4} \textbf{Left regular}: For each vertex $v\in \DirectedGraph^0$, on $\H_v^S$, $(\V,\S)$ is unitarily equivalent to an amplification of the left regular representation $(\lambda^\bN_v, \lambda^\DirectedGraph_v)^{(\alpha_v)}$, where $\alpha_v$ is the dimension of the space $\W_{\bN,v} = \ker V^* \cap \W_v^S$ and $\W_v^S = (S_v-\sum_{e\in vE^1} S_eS_e^*) \H_v^S$. Here, each unit vector in $\W_{\bN,v}$ is a wandering vector for $(\V,\S)$. 
\end{enumerate}
\end{thm}

\begin{proof} We have proved that each of these subspaces reduces both $\V$ and $\S$. Moreover, since $\S$ is a CK $\DirectedGraph$-family on $\H_C$, both \ref{item.Wold.main.1} and \ref{item.Wold.main.2} follow immediately.

For \ref{item.Wold.main.3}, $\S$ restricted on $\H_{[v]}^U$ is unitarily equivalent to $\bigoplus_{u\in\Omega_v} L_{\DirectedGraph,u}^{(\alpha_u)}$,
where $\alpha_u$ is the dimension of the wandering subspace $\W_u^U = (S_u-\sum_{e\in uE^1} S_eS_e^*)\H_u^U$. It suffices to show that $\alpha_u$ is constant among all $u\in\Omega_v$. By Proposition \ref{prop.TCK.reduceV}, $V\W_u^U \subset W_{1\cdot u}^U$ and $\W_{[v]}^U=\oplus_{u\in\Omega_v} \W_u^U$ reduces $V$. But since $V$ is unitary on $\H^U$, $V$ is also a unitary on $\W_{[v]}^U$ and thus $V$ is unitary from each $\W_u^U$ to $\W_{1\cdot u}^U$. Hence, the dimensions of $\W_u$ are constant among $u\in \Omega_v$. 

For \ref{item.Wold.main.4}, the orthogonal complement of the direct sum of the first three complements is $\H^S\cap \left(\bigoplus_{v\in\DirectedGraph^v} \H_v \right)$. Let $\H_v^S=\H^S \cap \H_v$ and $\W_v^S = (S_v-\sum_{e\in vE^1} S_eS_e^*) \H_v^S$. By Proposition \ref{prop.TCK.reduceV}, the space $\W^S = \oplus_{v\in\DirectedGraph^0} \W_v^S$ reduces $V$. Therefore, let $\W_{\bN,v}=\ker V^* \cap \W_v^S$ and we have that $\oplus_{v\in \DirectedGraph^0} \W_{\bN,v}=\ker V^* \cap \W^S$ and $\W^S = \bigoplus_{n\geq 0} V^n (\oplus_{v\in\DirectedGraph^*} \W_{\bN,v})$. Therefore, we have that 
\[\H^S \cap \left(\bigoplus_{v\in\DirectedGraph^0} \H_v \right) = \bigoplus_{\mu\in\DirectedGraph^*} S_\mu \W^S = \bigoplus_{v\in\DirectedGraph^0}\left(\bigoplus_{\mu\in\DirectedGraph^*}\bigoplus_{n\geq 0} S_\mu V^n \W_{\bN,v}\right).\]
Each unit vector in $\W_{\bN,v}$ is a wandering vector for $(\V,\S)$, and one can check that it generates a reducing subspace for $(\V,\S)$, on which $(\V,\S)$ is unitarily equivalent to a copy of $(\lambda_v^\bN, \lambda_v^\DirectedGraph)$.  Therefore, one can verify that $(\V,\S)$ restricted on each $\H_v^S=\bigoplus_{\mu\in\DirectedGraph^*}\bigoplus_{n\geq 0} S_\mu V^n \W_{\bN,v}$ is unitarily equivalent to $\dim \W_{\bN,v}$-copies of $(\lambda_v^\bN, \lambda_v^\DirectedGraph)$.
\end{proof}

For the unitary + pure shift type representations, we make the following key observation, which will be used later. 

\begin{lem}\label{lemma:key_lemma} Let $\S$ be a pure shift type TCK $\DirectedGraph$-family on a Hilbert space $\H$ and let \[\W=(I-\sum_{e\in E^1} S_e S_e^*)\H\] be its wandering space. 

Suppose $U_0$ is a unitary on $\W$ such that $U_0$ is a unitary map from $S_v \W$ to $S_{1\cdot v} \W$ for all $v\in E^0$. Then there exists a unique unitary + pure shift type representation $(\mathcal{U}, \S)$ on $\H$ such that $U|_\W = U_0$. 
\end{lem}

\begin{proof} From the Wold decomposition of $\S$, we have that $\H=\bigoplus_{\mu\in E^*} S_\mu \W$. For each $\mu\in E^*$ and $\xi\in\W$, define $U (S_\mu \xi) = S_{1\cdot \mu} U_0^{1|_\mu} \xi$. Now $S_\mu \xi=S_\mu S_{s(\mu)} \xi$ and by the assumption of $U_0$ and Condition~\ref{cond:S4}, we have $U_0^{1|_\mu} \xi \in S_{1\cdot s(\mu)} \xi$. Since $U_0$ is unitary from $S_{s(\mu)}\W$ to $S_{1\cdot s(\mu)} \W$, it is clear that $U$ is unitary from $S_\mu \W$ to $S_{1\cdot \mu} \W$, and thus $U$ is unitary on $\H$. Moreover, $U|_W=U_0$. One can easily verify that $U^k (S_\mu\xi)=S_{k\cdot \mu} U_0^{k|_\mu} \xi$ for $k\geq 1$. 

For each $v\in E^0$, $\mu\in E^*$, and $\xi\in \W$, we have
\begin{align*}
    U S_v S_\mu \xi &= \begin{cases} U S_\mu \xi, & \text{if }r(\mu)=v \\
    0, & \text{ otherwise}\end{cases} \\
    &= \begin{cases} S_{r(1\cdot \mu)} S_{1\cdot \mu} U_0^{1|_\mu} \xi, & \text{if }r(\mu)=v \\
    0, & \text{ otherwise}\end{cases} \\
    &= \begin{cases} S_{1\cdot r(\mu)}  U S_\mu \xi, & \text{if }r(\mu)=v \\
    0, & \text{ otherwise}\end{cases}\\
    &= S_{1\cdot v}  U S_\mu \xi.
\end{align*}
Since $\spn\{S_\mu\xi\}$ is dense in $\H$, we have $US_v=S_{1\cdot v} U$. 

For any $\mu, \nu\in E^*$ and $\xi\in \W$, we have
\begin{align*}
    US_\mu (S_\nu \xi) &= S_{1\cdot (\mu\nu)} U_0^{1_{\mu\nu}} \xi \\
    &= S_{1\cdot \mu} S_{1|_\mu \cdot \nu} U_0^{(1|_\mu)|_\nu} \xi \\
    &= S_{1\cdot \mu} U^{1_\mu} (S_\nu \xi).
\end{align*}
Therefore, $US_\mu=S_{1\cdot \mu} U^{1|_\mu}$. Since $U$ is unitary, we have by Proposition~\ref{prop.automatic.NC}, $(\U,\S)$ is a unitary + pure shift type representation on $\H$. The uniqueness follows from that $U$ is uniquely determined by the formula 
\[U (S_\mu \xi) = S_{1\cdot \mu} U^{1|_\mu} \xi = S_{1\cdot \mu} U_0^{1|_\mu} \xi. \qedhere \]
\end{proof}

\section{Atomic representations}

\begin{defn} A Toeplitz representation $(V,\S)$ of a self-similar graph $(\bN, E)$ on a Hilbert space $\H$ is called \emph{atomic} if there exists an orthonormal basis $\{\xi_i\}_{i\in I}$ for $\H$ satisfying the following properties: 
\begin{enumerate}
    \item There exists an injective map $\nu:I\to I$ such that for each $i\in I$, $V \xi_i=\lambda_i \xi_{\nu(i)}$ for some $\lambda_i\in \bT$.
    \item The set $I$ can be partitioned as a disjoint union of $I=\bigcup_{v\in E^0} I_v$ such that each $S_v$ is the orthogonal projection onto $\overline{\spn}\{\xi_i, i\in I_v\}$.
    \item For each $e\in E^1$, there exists an injective map $\pi_e:I_{s(e)} \to I_{r(e)}$ such that for each $i\in I_{s(e)}$,  $S_e \xi_i = \lambda_{e,i} \xi_{\pi_e(i)}$ for some $\lambda_{e,i}\in\bT$.
\end{enumerate}
\end{defn}

We call vectors $\xi_i$ \emph{atomic vectors} and the basis $\{\xi_i\}_{i\in I}$ an \emph{atomic basis}. Notice that the $S_e$'s have orthogonal ranges, and thus the ranges of $\{\pi_e\}_{e\in E^1}$ must be disjoint. 

If we only focus on the $\S$-part of the representation, then $\{S_v, S_e: v\in \DirectedGraph^0, e\in \DirectedGraph^1\}$ defines an atomic TCK $\DirectedGraph$-family as defined in \cite{DDL2019}. When the graph has a single vertex and $n$ edges, such representation was characterized by Davidson and Pitts in \cite{DP1999} in their study of free semigroup algebra. Their result was extended to atomic TCK $\DirectedGraph$-families in \cite{DDL2019}. Roughly speaking, every atomic TCK 
$\DirectedGraph$-family decomposes as a direct sum of one of the following three types:
\begin{enumerate}
    \item The left regular type: they are precisely $L_{E,v}$'s from Example \ref{ex.LEv}.
    \item The cycle type: when there exist $\xi_i$ and a directed path $\mu$ such that $S_\mu \xi_i=\lambda \xi_i$ for some scalar $\lambda\in \bT$.
    \item The inductive type: when every $\xi_i$ can be traced backwards in an infinite path. More precisely, we can find an infinite path $e_1e_2\cdots$ and distinct basic vectors $\xi_i=\xi_0, \xi_{-1}, \xi_{-2}, \cdots$, such that $S_{e_k} \xi_{-k} = \xi_{-k+1}$ for all $k\geq 1$. 
\end{enumerate}

Characterizing all atomic representations of a self-similar graph $(\mathbb{N},E)$ is a difficult task. We are only able to do so when the representation of the graph is pure. However, in the case of the odometer semigroups, we can fully characterize all atomic representations. 

\subsection{Atomic representations with pure shift $E$-representation}\label{sec.atomic.pureE}

Suppose that $(\V, \S)$ is a Toeplitz representation of a self-similar graph $(\mathbb{N},\DirectedGraph)$ where $\S$ is a pure shift. By Theorem~\ref{thm.Wold.main}, either $\V$ is unitary or $(\V,\S)$ is a direct sum of some amplifications of $(\lambda_v^\bN, \lambda_v^E)$. It suffices to consider the former case (since all representations of the latter case are already classified). 

We first define the archetypal atomic representation where $\V$ is unitary and $\S$ is pure. Let $u \in\DirectedGraph^0$ and let its orbit under the $\mathbb{N}$-action be $\Omega_v=\{u_1,u_2,\dots,u_m\}$, where $u=u_1$ and $1\cdot u_i=u_{i+1}$ (and by convention, $u_{m+1}=u_1$). 

Let $\H$ be the Hilbert space with an orthonormal basis $\spn\{\xi_{i,\mu}: s(\mu)=u_i, 1\leq i\leq m\}$. Define $S_e \xi_{i,\mu} = \xi_{i, e\mu}$ whenever $s(e)=r(\mu)$. One may notice that $\S$ is unitary equivalent to $\oplus_{i=1}^m L_{\DirectedGraph, u_i}$. 

Fix $\lambda\in\mathbb{T}$. First define a unitary $V$ on the wandering vectors $\{\xi_{i,u_i}: 1\leq i\leq m\}$ by
\[
V\xi_{i,u_i} = \begin{cases}
\xi_{i+1, u_{i+1}}, & \text{ if } 1\leq i<m \\
\lambda \xi_{1,\mu}, & \text{ if } i=m
\end{cases}
\]
By Lemma \ref{lemma:key_lemma}, 
there is a unique unitary $V$ on $\H$ such that $(\V,\S)$ is a unitary + pure representation. We call denote this representation by $c^\lambda_{u}=(V_{c,u}^\lambda, S_{c,u}^\lambda)$. It is not hard to see that for a fixed $u$, different values of $\lambda$ give non-equivalent atomic representations. Indeed, the wandering vectors are eigenvectors with eigenvalue $\lambda$ for $c^\lambda_u$. 

If $u,v$ are in the same orbit, $c^\lambda_u$ and $c^\lambda_v$ are actually unitarily equivalent. This can be done by re-scaling the wandering vectors appropriately. Therefore, with a slight abuse of notation, we use $c^\lambda_u$ to denote the class of atomic representations $c^\lambda_v$, $v\in\Omega_u$.  

Now, the main goal for the rest of this subsection is to prove that any unitary + pure shift atomic representation is a direct sum (or integral) of these $c^\lambda_u$. In other words, these are all the irreducible atomic representations of the unitary + pure shift type. 

\begin{prop} 
\label{P:atomic}
Let $(\V,\S)$ be a unitary + pure shift type atomic representation of a self-similar graph $(\mathbb{N},\DirectedGraph)$ on a Hilbert space $\H$. Then $(\V,\S)$ decomposes as a direct sum or direct integral of some $c^\lambda_u$'s. 
\end{prop}

\begin{proof} For $u\in \DirectedGraph^0$, let $\W_u=(S_u-\sum_{e\in u\DirectedGraph^1} S_eS_e^*)\H$. 
For each equivalence class $\Omega_v$, define $\W_{[v]}=\oplus_{u\in\Omega_v} \W_u$ and $\H_{[v]}=\bigoplus_{\mu\in \DirectedGraph^*} S_\mu \W_{[v]}$. By Theorem~\ref{thm.Wold.main}, each $\H_{[v]}$ reduces $(\V,\S)$ and on $\H_{[v]}$, $\S$ is unitarily equivalent to $\left(\bigoplus_{u\in\Omega_v} L_{\DirectedGraph,u}\right)^{(\alpha)}$ where $\alpha=\dim \W_u$. 

It suffices to prove that on each $\H_{[v]}$, $(\V,\S)$ is a direct sum or integral of some $c^\lambda_v$. We first consider the case when $\alpha<\infty$. Let $\Omega_v=\{v_1,v_2,\ldots,v_m\}$ where $1\cdot v_i=v_{i+1}$. Pick an atomic orthonormal basis for $\H_{[v]}$. 
Since such an orthonormal basis also makes $\S$ an atomic representation, we can always pick some basic vectors $\{\xi_{i,j}: 1\leq j\leq \alpha\}$ that spans $\W_{v_i}$. 
Without loss of generality, $V$ cycles through this entire basis (i.e. $\spn\{V^n \xi\}=\oplus_{i=1}^m \W_{v_i}$ for any basic vector $\xi$ among $\xi_{i,j}$). Otherwise, each $\V$-orbit on $\xi_{i,j}$'s generates a reducing subspace for $(\V,\S)$. 

Since $V\W_v=\W_{1\cdot v}$, we have $V\xi_{i,j}=\lambda_{i,j}\xi_{i+1, j'}$ for some $j'$ and $\lambda_{i,j}\in\mathbb{T}$. We can always reorder the $j$'s so that $j'=j$ when $1\leq i<m$; and $j'=j+1$ when $i=m$. We can also re-scale the orthonormal basis so that $\lambda_{i,j}=1$ except $\lambda_{m,\alpha}=\lambda\in\mathbb{T}$. As an illustration, this is summarized in the following diagram.  

\begin{figure}[h]
    \centering

    \begin{tikzpicture}[scale=1, every node/.style={scale=1}]
    \foreach \x in {0,1,3}{
    \foreach \y in {0,1,3}{
        \node at (\x*2, \y*2) {$\bullet$};
        \node at (4, \y*2) {$\cdots$};
    }
    \node at (\x*2, 4) {$\vdots$};
    }
    \draw  [-{Stealth[length=2mm, width=1mm]}] plot [smooth] coordinates {(0,0) (1,-0.1) (2,0)};
    \draw  [-{Stealth[length=2mm, width=1mm]}] plot [smooth] coordinates {(2,0) (2.7,-0.1) (3.5,0)};
    \draw  [-{Stealth[length=2mm, width=1mm]}] plot [smooth] coordinates {(4.5,0) (5.2,-0.1) (6,0)};
    \draw  [-{Stealth[length=2mm, width=1mm]}] plot [smooth] coordinates {(6,0) (5.9, 0.5) (0.1,1.5) (0,2)};
    
    \draw  [-{Stealth[length=2mm, width=1mm]}] plot [smooth] coordinates {(0,2) (1,2-0.1) (2,2)};
    \draw  [-{Stealth[length=2mm, width=1mm]}] plot [smooth] coordinates {(2,2) (2.7,2-0.1) (3.5,2)};
    \draw  [-{Stealth[length=2mm, width=1mm]}] plot [smooth] coordinates {(4.5,2) (5.2,2-0.1) (6,2)};
    \draw  [-{Stealth[length=2mm, width=1mm]}] plot [smooth] coordinates {(6,2) (5.9, 2.5) (3,3)};
    
    \draw  [-{Stealth[length=2mm, width=1mm]}] plot [smooth] coordinates {(0,6) (1,6-0.1) (2,6)};
    \draw  [-{Stealth[length=2mm, width=1mm]}] plot [smooth] coordinates {(2,6) (2.7,6-0.1) (3.5,6)};
    \draw  [-{Stealth[length=2mm, width=1mm]}] plot [smooth] coordinates {(4.5,6) (5.2,6-0.1) (6,6)};
    \draw  [-{Stealth[length=2mm, width=1mm]}] plot [smooth] coordinates {(3,5) (0.1,5.5) (0,6)};
    
    \draw plot [smooth] coordinates {(6,6) (5.9, 6.5) (0.1,6.5)};
    \draw plot [smooth] coordinates {(0.1,6.5) (-0.42,6.32) (-0.5,5.9)};
    \draw  [-{Stealth[length=2mm, width=1mm]}] plot [smooth] coordinates {(-0.5,5.9) (-0.5,0.1) (0,0)};
    
    \node at (0.2, -0.3) {$\xi_{1,1}$};
    \node at (2.2, -0.3) {$\xi_{2,1}$};
    \node at (6.2, -0.3) {$\xi_{m,1}$};
    
    \node at (0.2, 2-0.3) {$\xi_{1,2}$};
    \node at (2.2, 2-0.3) {$\xi_{2,2}$};
    \node at (6.2, 2-0.3) {$\xi_{m,2}$};
    
    \node at (0.2, 6-0.3) {$\xi_{1,\alpha}$};
    \node at (2.2, 6-0.3) {$\xi_{2,\alpha}$};
    \node at (6.2, 6-0.3) {$\xi_{m,\alpha}$};
    
    \node at (-0.5,6.5) {$\lambda$};
    \end{tikzpicture}
    \label{fig:B.sB}
\end{figure}

Fix $\beta\in\mathbb{T}$ where $\beta^\alpha=\overline{\lambda}$, and for each $\alpha$-root of unity $\omega_k$ (i.e. $\omega_k^\alpha=1$), $1\leq k\leq\alpha$, define 
\[\eta_{k,1} = \sum_{j=1}^\alpha (\beta \omega_k)^{j-1} \xi_{1,j}.\]
For each $1<i\leq m$, let $\eta_{k,i} = V^{i-1} \eta_{k,1}$. We can compute that 
\begin{align*}
V\eta_{k,m} &= V^m \eta_{k,1} \\
&= \sum_{j=1}^{\alpha-1} (\beta \omega_k)^{j-1} \xi_{1,j+1} + (\beta\omega_k)^{\alpha-1} \lambda \xi_{1,1}  \\
&= \frac{1}{\beta\omega_k} \eta_{1,1}.
\end{align*}
Therefore, by setting $\H_{k}=\overline{\spn}\{S_\mu \eta_{k,i}: 1\leq i\leq m, \mu\in\DirectedGraph^*\}$, we have $\H_k$ reduces $(\V,\S)$, on which $(\V,\S)$ is unitarily equivalent to the atomic representation $c^{\overline{\beta\omega_k}}_v$. It is routine to check $\H_k$ ($1\leq k\leq \alpha$) are pairwise orthogonal to each other and their direct sum is $\H_{[v]}$. Therefore, $(\V,\S)=\oplus_{k=1}^\alpha c^{\overline{\beta\omega_k}}_v$. 

For the case when $\alpha=\infty$, one can follow a similar argument to prove that $(\V,\S)=\int^\oplus_{\beta\in \mathbb{T}} c^\beta_v d\beta$. 
\end{proof}

\subsection{Atomic representations of odometer semigroups}

The classification of all atomic representations of self-similar graphs is difficult. However, we can character all irreducible atomic representations in the case of odometer semigroups studied in \cite{BLi2022}. 

Let $(\V,\S)$ be an atomic representation of the odometer semigroup $\mathbb{O}_n$. As a reminder, $\S$ is given by a row isometry $(S_1,\ldots,S_n)$ and $\V$ is determined by a single isometry $V$. They satisfy the relations 
\begin{align*}
    V S_k &= S_{k+1}, \ 1\leq k < n, \\
    V S_n &= S_1 V, \\
    V^* S_1 &= S_n V^*.
\end{align*}

The characterization when $\S$ is pure (i.e., Proposition \ref{P:atomic}) applies to all Toeplitz representations of self-similar graphs. It remains to consider the cases when $\S$ is a CK $\DirectedGraph$-family, in which case, $\sum_{i=1}^n S_i S_i^*=I$. 
By \cite[Theorem 3.4]{DP1999}, there are two possibilities for such $\S$. 
\begin{enumerate}
    \item $\S$ is of the inductive type: there is an infinite sequence of distinct basic vectors $\{\xi_{-n}: n\geq 0\}$ and an infinite word $\cdots e_n e_{n-1} \cdots e_1$ such that $S_{e_n} \xi_{-n} = \xi_{-(n-1)}$. Here the infinite word is aperiodic.
    \item $\S$ is of the cycle type: there exists a primitive word $e_1e_2\cdots e_n$, basic vectors $\{\xi_i: 1\leq i\leq n\}$, and a scalar $\lambda\in\mathbb{T}$ such that $S_{e_i} \xi_i = \xi_{i+1}$ for $1\leq i<n$ and $S_{e_n} \xi_n=\lambda \xi_1$. 
\end{enumerate}
Here, an infinite word is aperiodic if it is not tail equivalent to an infinite periodic word. A finite word is called primitive if it is not a power of some smaller words. 

We first make the following observation. Suppose $\mu$ is not all $n$, then $1|_\mu=0$. Indeed, by Condition~\ref{item:S3}, $1|_{n\mu}=(1|_n)|_\mu=1_\mu$ and $1|_{i\mu}=(1|_i)|_\mu=0|_\mu=0$ for $i\neq n$. Therefore, $1|_\mu=0$ as long as $\mu$ is not all $n$. 

Now suppose $\xi\in\H$ has $\xi=S_\mu \xi_0$ for some $\mu$ that is not all $n$, then $V\xi = S_{1\cdot \mu} \xi_0$. Therefore, the value of $V$ is uniquely determined for such a vector $\xi$. Notice that when $\S$ is an irreducible inductive type atomic representation, each basic vector $\xi$ is in the range of some $S_\mu$ where $\mu$ is not all $n$. Therefore, $V\xi$ is uniquely determined in such a case. We observe that this uniquely determines a unitary + CK type atomic representation.

\begin{prop}\label{prop.inductive} When $\S$ is an irreducible inductive type atomic representation on $\H$ with respect to the orthonormal basis $\{\xi_i\}_{i\in I}$. There exists a unique unitary $V$ on $\H$ such that $(\V,\S)$ is a unitary + CK type atomic representation. 
\end{prop}

\begin{proof} Since $S$ is of the inductive type, each basic vector $\xi$ can be written as 
\begin{align*}
    \xi = S_{e_1} \xi_{-1} 
     = S_{e_1} S_{e_2} \xi_{-2} 
    = \dots 
    = S_{e_1} \cdots S_{e_k} \xi_{-k}.
\end{align*}
Here, $\xi_{-n}$'s are all basic vectors and the infinite word $e_1 e_2 \cdots$ is aperiodic. Suppose $s$ is the smallest index such that $e_s\neq n$, then we have
\begin{align*}
    V\xi &= S_1 \cdots S_1 S_{e_s+1} \xi_{-s} \\
    &= S_1 \cdots S_1 S_{e_s+1} S_{e_{s+1}} \xi_{-(s+1)} \\
    &= S_1 \cdots S_1 S_{e_s+1} S_{e_{s+1}} \cdots S_{e_k} \xi_{-k}.
\end{align*}
Therefore, whenever $\xi=S_\mu \xi_0$ for some $\mu$ not all $n$, we always have $V\xi = S_{1\cdot \mu} \xi_0$, and this is well-defined. Since $S$ is an atomic representation, $S_{1\cdot \mu}\xi_0$ is also an atomic basic vector. Notice that the $\mathbb{N}$-action $1\cdot\mu$ is an automorphism of $E^*$, we have that $V$ is an injective map on the atomic basic vectors. 

To see that $V$ is unitary, we observe that by the infinite aperiodic word associated with each $\xi$, we can find $\xi=S_1\cdots S_1 S_{e_s} \xi_{-s}$ for some $e_s\neq 1$. Then let $\xi_0=S_1\cdots S_1 S_{e_s-1} \xi_{-s}$. We have $V\xi_0=\xi$, and thus $V$ is a bijection on the atomic basic vectors and hence is unitary. 

To see that $(\V,\S)$ is a unitary + CK type atomic representation, we are left to check the Toeplitz condition. For any basic vector $\xi$, suppose $\xi=S_n\cdots S_n S_{e_s} \xi_{-s}$ where $e_s\neq n$. We have
\begin{align*}
     VS_i \xi &= \begin{cases}
     S_{i+1}\xi, & \text{if } i\neq n\\
     VS_n (S_n\cdots S_n S_{e_s} \xi_{-s}), & \text{if } i=n
     \end{cases} \\
     &= \begin{cases}
     S_{i+1}\xi, & \text{if } i\neq n\\
     S_1 (S_1\cdots S_1 S_{e_s+1} \xi_{-s}), & \text{if } i=n
     \end{cases} \\
     &= \begin{cases}
     S_{i+1}\xi, & \text{if } i\neq n\\
     S_1 V(S_n\cdots S_n S_{e_s} \xi_{-s}), & \text{if } i=n
     \end{cases} \\
     &= \begin{cases}
     S_{i+1}\xi, & \text{if } i\neq n\\
     S_1 V \xi, & \text{if } i=n.
     \end{cases} 
\end{align*}
Since $V$ is unitary and $\S$ is CK, we have that $(\V,\S)$ is indeed a Toeplitz representation. 
\end{proof}

We are now left with the case when $\S$ is a cycle type. Observe that if the primitive word for the cycle $e_1e_2\cdots e_n$ contains anything other than $1$ or $n$, then each $\xi=S_\mu\xi_1$ for some $\mu$ that is not all $n$ or $1$. Therefore, the same proof for the Proposition~\ref{prop.inductive} applies. In such a case, there exists a unique unitary $V$ such that $(\V,\S)$ is a unitary + CK type atomic representation. 

We are left with two cases: either we have a 1-cycle of $1$ for the atomic representation $S$ or a 1-cycle of $n$. We deal with the two cases separately here.

\begin{prop}\label{prop.1cycle.case1}  Suppose $(\V,\S)$ is an atomic representation on $\H$ and suppose $\{\xi_i\}_{i\in I}$ are atomic basic vectors such that $\S$ is a cycle-type representation with a 1-cycle of $n$ on $\xi_0$ (i.e. $S_n \xi_0=\lambda \xi_0$, $\lambda\in\mathbb{T}$). 
Then there exists basic vectors $\{\eta_i\}_{i\in I}$ that are disjoint from $\xi_i$'s, such that $(\V,\S)$ is a unitary + CK type atomic representation on the reducing subspace $\overline{\spn}\{\xi_i, \eta_i\}$. 
\end{prop}

\begin{proof} 
Let $V\xi_0=\omega\eta_0$ ($\omega\in\mathbb{T}$), which is a basic atomic vector since $(\V,\S)$ is atomic. Without loss of generality, by re-scaling the basic atomic vectors, we may assume $\omega=1$. Now,
\[S_1 \eta_0 = S_1 V \xi_0 = V S_n \xi_0 = V (\lambda \xi_0) = \lambda \eta_0.\]
Therefore, $\S$ has a 1-cycle of $1$ on $\eta_0$. Now let $\{\eta_i\}_{i\in I}=\{S_\mu \eta_0: \mu\text{ does not end with }1\}$ are all basic vectors and we know from the atomic representation of $\S$ that $\S$ is a cycle-type atomic representation on the space $\overline{\spn}\{\eta_i\}$. For each $\eta_i=S_\mu \eta_0$, either $\mu$ is not all $n$, or when it is, we can write $\eta_i=S_\mu (\overline{\lambda}S_1 \eta_0)$. In either case, $V\eta_i$ is uniquely determined following the same argument in Proposition~\ref{prop.inductive}. 

To see $V$ is unitary, the only vector among $\{\xi_i, \eta_i\}_{i\in I}$ that cannot be written as a $S_\mu e$ with $\mu$ not all $1$ is the vector $\eta_0$. But $\eta_0=V \xi_0\in\ran(V)$ as well. Therefore, $V$ is unitary on the reducing subspace $\overline{\spn}\{\xi_i, \eta_i\}$.
Therefore, $(\V,\S)$ is a unitary + CK type atomic representation on the reducing subspace $\overline{\spn}\{\xi_i, \eta_i\}$. 
\end{proof}

\begin{prop}\label{prop.1cycle.case2} Suppose $(\V,\S)$ is an atomic representation on $\H$ and suppose $\{\eta_i\}_{i\in I}$ are atomic basic vectors such that $\S$ is a cycle-type representation with a 1-cycle of $1$ on $\eta_0$ (i.e. $S_1 \eta_0=\lambda \eta_0$, $\lambda\in\mathbb{T}$). Then there are two cases:
\begin{enumerate}
    \item If $\eta_0\notin \ran V$, then $(\V,\S)$ is a pure + CK type atomic representation on the reducing subspace $\overline{\spn}\{\eta_i\}$. 
    \item If $\eta_0\in \ran V$. Then there exists basic atomic vectors $\{\xi_i\}_{i\in I}$ that are disjoint from $\eta_i$'s, such that $(\V,\S)$ is a unitary + CK type atomic representation on the reducing subspace $\overline{\spn}\{\xi_i, \eta_i\}$. In particular, $\S$ when restricted on $\overline{\spn}\{\xi_i\}$ is a cycle-type representation with a 1-cycle of $n$ on $V^* \eta_0$. 
\end{enumerate}
\end{prop}

\begin{proof} Since each vector $\eta_i$ can be written as $\eta_i=S_\mu \eta_0=\lambda S_\mu S_1 \eta_0$, we know $V\eta_i$ is uniquely determined by the same argument as in Proposition~\ref{prop.inductive}. 

If $\eta_0\notin \ran V$, then $\overline{\spn}\{\eta_i\}$ reduces both $V$ and $S$ and it is clear that $(\V,\S)$ is a pure + CK type atomic representation. 

If $\eta_0\in\ran V$, then $\eta_0=V\xi_0$. Then we have
\[S_n \xi_0 = S_n V^* \eta_0 = V^* S_1 \eta_0 = \lambda V^* \eta_0 = \lambda \xi_0.\]
Therefore, $\S$ has a 1-cycle of $1$ on $\eta_0$. Now let $\{\xi_i\}_{i\in I}=\{S_\mu \xi_0: \mu\text{ does not end with }n\}$. We know $\S$ is a cycle-type atomic representation on the space $\overline{\spn}\{\xi_i\}$ with a 1-cycle of $n$ on $\xi_0$. This is essentially the same situation as in Proposition~\ref{prop.1cycle.case1}, in which case we get a unitary + CK type atomic representation on $\overline{\spn}\{\xi_i, \eta_i\}$. 
\end{proof}

\begin{eg} The following diagram illustrates an example of unitary + CK type atomic representation for the odometer semigroup $\bO_2$ where $\S$ is a cycle-type representation. 

\begin{figure}[h]
    \centering
    
    \def\gspace{5}
    \begin{tikzpicture}[scale=1, every node/.style={scale=1}]
    
    \node at (0,4) {$\bullet$};
    
    \node at (0,2) {$\bullet$};
    \node at (-1.5,0) {$\bullet$};
    \node at (1.5,0) {$\bullet$};
    
    \node at (-2,-2) {$\bullet$};
    \node at (-1,-2) {$\bullet$};
    \node at (1,-2) {$\bullet$};
    \node at (2,-2) {$\bullet$};
    
    \draw [decoration={markings,
    mark=at position 0.97 with \arrow{>}},
    postaction=decorate] plot [smooth] coordinates {(0,4) (0.3, 4.6) (0, 4.8) (-0.3,4.6) (0,4)};
    \node at (0,5) {$1$};
    
	\draw[decoration={markings,
    mark=at position 0.97 with \arrow{>}},
    postaction=decorate] (0,4) -- (0,2);
	\node at (0.2,3) {$2$};
	
	\draw[decoration={markings,
    mark=at position 0.97 with \arrow{>}},
    postaction=decorate] (0,2) -- (-1.5,0);
    \node at (-1,1) {$1$};
	\draw[decoration={markings,
    mark=at position 0.97 with \arrow{>}},
    postaction=decorate] (0,2) -- (1.5,0);
    \node at (1,1) {$2$};
    
    \draw[decoration={markings,
    mark=at position 0.97 with \arrow{>}},
    postaction=decorate] (-1.5,0) -- (-2,-2);
    \node at (-1.9,-1) {$1$};
	\draw[decoration={markings,
    mark=at position 0.97 with \arrow{>}},
    postaction=decorate] (-1.5,0) -- (-1,-2);
    \node at (-1.1,-1) {$2$};
    
    \draw[decoration={markings,
    mark=at position 0.97 with \arrow{>}},
    postaction=decorate] (1.5,0) -- (1,-2);
    \node at (1.1,-1) {$1$};
	\draw[decoration={markings,
    mark=at position 0.97 with \arrow{>}},
    postaction=decorate] (1.5,0) -- (2,-2);
    \node at (1.9,-1) {$2$};

    \draw [dotted, decoration={markings, post length=1mm,
                 pre length=1mm,
    mark=at position 0.5 with \arrow{>}},
    postaction=decorate] plot [smooth] coordinates {(0,4) (-0.4,3) (0,2)};
    
    \draw [dotted, decoration={markings, post length=1mm,
                 pre length=1mm,
    mark=at position 0.5 with \arrow{>}},
    postaction=decorate] plot [smooth] coordinates {(0,2) (-1.25, 1.5) (-1.5,0)};
    
    \draw [dotted, decoration={markings, post length=1mm,
                 pre length=1mm,
    mark=at position 0.5 with \arrow{>}},
    postaction=decorate] plot [smooth] coordinates {(-1.5,0) (0, 0.4) (1.5,0)};
	
    \draw [dotted, decoration={markings, post length=1mm,
                 pre length=1mm,
    mark=at position 0.5 with \arrow{>}},
    postaction=decorate] plot [smooth] coordinates {(1.5,0) (-0.5, -0.4) (-2,-2)};
    
    \draw [dotted, decoration={markings, post length=1mm,
                 pre length=1mm,
    mark=at position 0.5 with \arrow{>}},
    postaction=decorate] plot [smooth] coordinates {(-2,-2) (-1.5, -1.8) (-1,-2)};
    
    \draw [dotted, decoration={markings, post length=1mm,
                 pre length=1mm,
    mark=at position 0.5 with \arrow{>}},
    postaction=decorate] plot [smooth] coordinates {(-1,-2) (0, -1.6) (1,-2)};
    
    \draw [dotted, decoration={markings, post length=1mm,
                 pre length=1mm,
    mark=at position 0.5 with \arrow{>}},
    postaction=decorate] plot [smooth] coordinates {(1,-2) (1.5, -1.8) (2,-2)};
    
    \draw [dotted, decoration={markings, post length=1mm,
                 pre length=1mm,
    mark=at position 0.97 with \arrow{>}},
    postaction=decorate] plot [smooth] coordinates {(2,-2) (1, -2.2) (0,-2.6)};
    
    \node at (-1.5,-2.5) {$\vdots$};
    \node at (1.5,-2.5) {$\vdots$};
    

    \node at (0+\gspace,4) {$\bullet$};

    \node at (0+\gspace+0.35,4) {$\xi_0$};
    \node at (0-+0.35,4) {$\eta_0$};
    
    \node at (0+\gspace,2) {$\bullet$};
    \node at (-1.5+\gspace,0) {$\bullet$};
    \node at (1.5+\gspace,0) {$\bullet$};
    
    \node at (-2+\gspace,-2) {$\bullet$};
    \node at (-1+\gspace,-2) {$\bullet$};
    \node at (1+\gspace,-2) {$\bullet$};
    \node at (2+\gspace,-2) {$\bullet$};
    
    \draw [decoration={markings,
    mark=at position 0.97 with \arrow{<}},
    postaction=decorate] plot [smooth] coordinates {(0+\gspace,4) (0.3+\gspace, 4.6) (0+\gspace, 4.8) (-0.3+\gspace,4.6) (0+\gspace,4)};
    \node at (0+\gspace,5) {$2$};
    
	\draw[decoration={markings,
    mark=at position 0.97 with \arrow{<}},
    postaction=decorate] (0+\gspace,4) -- (0+\gspace,2);
	\node at (0.2+\gspace,3) {$1$};
	
	\draw[decoration={markings,
    mark=at position 0.97 with \arrow{<}},
    postaction=decorate] (0+\gspace,2) -- (-1.5+\gspace,0);
    \node at (-1+\gspace,1) {$2$};
	\draw[decoration={markings,
    mark=at position 0.97 with \arrow{<}},
    postaction=decorate] (0+\gspace,2) -- (1.5+\gspace,0);
    \node at (1+\gspace,1) {$1$};
    
    \draw[decoration={markings,
    mark=at position 0.97 with \arrow{<}},
    postaction=decorate] (-1.5+\gspace,0) -- (-2+\gspace,-2);
    \node at (-1.9+\gspace,-1) {$2$};
	\draw[decoration={markings,
    mark=at position 0.97 with \arrow{<}},
    postaction=decorate] (-1.5+\gspace,0) -- (-1+\gspace,-2);
    \node at (-1.1+\gspace,-1) {$1$};
    
    \draw[decoration={markings,
    mark=at position 0.97 with \arrow{<}},
    postaction=decorate] (1.5+\gspace,0) -- (1+\gspace,-2);
    \node at (1.1+\gspace,-1) {$2$};
	\draw[decoration={markings,
    mark=at position 0.97 with \arrow{<}},
    postaction=decorate] (1.5+\gspace,0) -- (2+\gspace,-2);
    \node at (1.9+\gspace,-1) {$1$};

    \draw [dotted, decoration={markings, post length=1mm,
                 pre length=1mm,
    mark=at position 0.5 with \arrow{<}},
    postaction=decorate] plot [smooth] coordinates {(0+\gspace,4) (-0.4+\gspace,3) (0+\gspace,2)};
    
    \draw [dotted, decoration={markings, post length=1mm,
                 pre length=1mm,
    mark=at position 0.5 with \arrow{<}},
    postaction=decorate] plot [smooth] coordinates {(0+\gspace,2) (-1.25+\gspace, 1.5) (-1.5+\gspace,0)};
    
    \draw [dotted, decoration={markings, post length=1mm,
                 pre length=1mm,
    mark=at position 0.5 with \arrow{<}},
    postaction=decorate] plot [smooth] coordinates {(-1.5+\gspace,0) (0+\gspace, 0.4) (1.5+\gspace,0)};
	
    \draw [dotted, decoration={markings, post length=1mm,
                 pre length=1mm,
    mark=at position 0.5 with \arrow{<}},
    postaction=decorate] plot [smooth] coordinates {(1.5+\gspace,0) (-0.5+\gspace, -0.4) (-2+\gspace,-2)};
    
    \draw [dotted, decoration={markings, post length=1mm,
                 pre length=1mm,
    mark=at position 0.5 with \arrow{<}},
    postaction=decorate] plot [smooth] coordinates {(-2+\gspace,-2) (-1.5+\gspace, -1.8) (-1+\gspace,-2)};
    
    \draw [dotted, decoration={markings, post length=1mm,
                 pre length=1mm,
    mark=at position 0.5 with \arrow{<}},
    postaction=decorate] plot [smooth] coordinates {(-1+\gspace,-2) (0+\gspace, -1.6) (1+\gspace,-2)};
    
    \draw [dotted, decoration={markings, post length=1mm,
                 pre length=1mm,
    mark=at position 0.5 with \arrow{<}},
    postaction=decorate] plot [smooth] coordinates {(1+\gspace,-2) (1.5+\gspace, -1.8) (2+\gspace,-2)};
    
    \draw [dotted, decoration={markings, post length=1mm,
                 pre length=1mm,
    mark=at position 0.97 with \arrow{<}},
    postaction=decorate] plot [smooth] coordinates {(2+\gspace,-2) (1+\gspace, -2.2) (0+\gspace,-2.6)};
    
     \draw [dotted, decoration={markings, post length=1mm,
                 pre length=1mm,
    mark=at position 0.97 with \arrow{>}},
    postaction=decorate] plot [smooth] coordinates {(\gspace,4) (0,4)};
    
    \node at (-1.5+\gspace,-2.5) {$\vdots$};
    \node at (1.5+\gspace,-2.5) {$\vdots$};
    
    \end{tikzpicture}
\end{figure}

If we restrict to the invariant subspace generated by the basic vectors on the left side of the diagram, we obtain a pure + CK type atomic representation.     
\end{eg}

\section{Dilation and Maximal Representations}

Given a non-selfadjoint operator algebra $\mathcal{A}$, there exists a smallest C*-algebra that $\mathcal{A}$ embeds in, which is called the C*-envelope of $\A$. One way of computing C*-envelopes is by characterizing all the maximal representations using dilation theory techniques. Let $\rho:\mathcal{A}\to\mathcal{B}(\H)$ be a (completely contractive) representation of a non-selfadjoint operator algebra $\mathcal{A}$. A dilation of $\rho$ is a representation $\pi:\mathcal{A}\to\mathcal{B}(\mathcal{K})$ where $\H\subset\K$ and $\rho(a)=P_\H \pi(a) |_\H$ for all $a\in \mathcal{A}$. The representation $\rho$ is called \emph{maximal} if any dilation $\pi$ is in the form $\pi=\rho\oplus \sigma$. 

Dritschel and McCollough \cite{DM2005} proved that the C*-envelope of $\mathcal{A}$ is the C*-algebra generated by completely isometric maximal representations. Therefore, computing C*-envelopes is the same as characterizing all maximal representations. 

Non-selfadjoint operator algebras arise naturally from semigroup dynamics \cite{DFK2014, BLi2021}. A natural object to consider is universal non-selfadjoint operator algebras, which are universal with respect to certain representations. In our case, given a self-similar graph $(\mathbb{N},\DirectedGraph)$, it is associated with a small category $\DirectedGraph^* \bowtie P=\{(\mu, p): \mu\in E^*, p\in P\}$.
Any Toeplitz representation $(\V,\S)$ defines a representation of the small category $E^*\bowtie P$ by $\V\times\S((\mu,p))=S_\mu V_p$. We can build a universal non-selfadjoint operator algebra $\A_{P, E}$ by taking the completion of the (non-selfadjoint) operator algebra generated by the universal Toeplitz representation $\oplus (\V\times\S)$. This direct sum is not over an empty set because the left-regular representation is always a Toeplitz representation.
We call this 
non-selfadjoint operator algebra the universal non-selfadjoint operator algebra $\A_{P,\DirectedGraph}$  for the self-similar graph. 

In this section, we fix a self-similar graph $(\mathbb{N}, E)$ and consider its universal non-selfadjoint operator algebra $\A_{\mathbb{N},\DirectedGraph}$. The Wold decomposition (Theorem~\ref{thm.Wold.main}) implies that the representation of $\A_{\mathbb{N},\DirectedGraph}$ is a direct sum of four components. The main goal of this section is to compute the C*-envelope of $\A_{\mathbb{N},\DirectedGraph}$ by identifying which component in its Wold decomposition is maximal. We will prove the following theorem.

\begin{thm}\label{thm.main.envelope} The C*-envelope of $\A_{\mathbb{N},\DirectedGraph}$ is the self-similar graph C*-algebra $\O_{\mathbb{Z},\DirectedGraph}$, which is also the universal C*-algebra generated by a unitary + CK type Toeplitz representation. 
\end{thm}

By the result of Dritschel and McCollough \cite{DM2005}, it suffices to prove that a Toeplitz representation is maximal if and only if it is a unitary + CK type. Showing unitary + CK type Toeplitz representations are always maximal is the easier direction.

\begin{prop}\label{P:UU} Let $(\V,\S)$ be a unitary + CK type representation of a self-similar graph $(\mathbb{N},\DirectedGraph)$. Then $(\V,\S)$ is a maximal representation. 
\end{prop}

\begin{proof}
Suppose that $(\V', \S')$ is an isometric Nica-covariant dilation of $(\V, \S)$. Then 
\begin{align*}
V'=
\begin{pmatrix}
V& A\\
0& B
\end{pmatrix}, \quad 
S_v'=
\begin{pmatrix}
S_v & 0\\
0 & P_v
\end{pmatrix} \qforal v\in E^0,
\quad
S_e'=
\begin{pmatrix}
S_e & C_e\\
E_e & D_e
\end{pmatrix} \qforal e\in E^1. 
\end{align*}
We first notice that $S_e'^* S_e'= S_{s(e)}'$, so that $S_e^* S_e + E_e^* E_e = S_{s(e)}$ and thus $E_e=0$. 

Then 
\begin{align*}
V'V'^*=
\begin{pmatrix}
VV^* + AA^*& AB^*\\
BA^*& B B^*
\end{pmatrix}, \quad 
S_e' S_e'^*=
\begin{pmatrix}
S_e S_e^*+ C_eC_e^* & C_e D_e^*\\
D_e C_e^* & D_eD_e^*
\end{pmatrix} \qforal e\in E^1. 
\end{align*}
But $V'V'^*\le I$ and $VV^*=I$. So $A=0$. 
For $\S'$, on one hand, $\sum_{e\in E^1} S_e'S_e'^*\le I$, but on the other hand $\sum_{e\in E^1} S_eS_e^*= I$ as $\S$ is CK. 
Hence $C_e=0$ for all $e\in E^1$. 
Therefore, the representation $(\V,\S)$ is maximal. 
\end{proof}

Thus it suffices to prove that each of the remaining cases, i.e., Theorem~\ref{thm.Wold.main} (ii) - (iv), 
admits a non-trivial dilation and thus cannot be maximal. 

\subsection{$\V$ is a pure isometry}

We first deal with the case where $\V$ is a pure isometry. This includes the two cases -- pure shift + CK type and left regular representations
(i.e., Theorem~\ref{thm.Wold.main} (ii) and (iv)). 


The following property from the assumption (A) is recorded for later frequent use. 

\begin{lem}
\label{L:AA}
For each $e\in \DirectedGraph^1$, let $m_e:=|\Omega_e|$, the cardinality of $\Omega_e$. 
\begin{itemize}
\item[(i)] 
For every $e\in E^1$ and $k\in \bN$, there are $f\in E^1$, $p,q\in \bN$ with $p\ge k$ such that 
\begin{align*}
q\cdot e=f \text{ and }q|_e=p\ge k.
\end{align*}

\item[(ii)]
If $(\V, \S)$ is a representation of $(\bN, E)$, then $V^{n m_e} S_e =S_e V^{n \sum_{f\in \Omega_e} 1|_f}$ for every $n\in \bN$ and $e\in E^1$.
\end{itemize}
\end{lem}

\begin{proof}
(i) This follows from the fact that, for fixed $e\in E^1$, $n|_e$ is increasing with respect to $n\in \bN$. 

For (ii), it suffices to notice that $m_e\cdot e=e$ and $m_e|_e=\sum_{f\in \Omega_e} 1|_f$. 
\end{proof}

Let $\W:=\ker V^*$ and define 
\[
\widehat \H=\W\otimes \ell^2(\bZ).
\]
Let $J: \H\to \widehat \H$ be the operator $V^k\xi\mapsto \xi\otimes e_k$ for all $\xi\in \W$ and $k\ge 0$. 
Then $\H$ is embeded  into $\widehat\H$ via $J(\H)=\W\otimes \ell^2(\bN)$.

The isometry $V$ on $\H$ can now be easily extended to a unitary $\widehat V$ on $\widehat\H$ as follows
\[
\widehat V: \widehat \H\to \widehat \H,\ \xi\otimes e_n\mapsto \xi\otimes e_{n+1}\qforal n\in \bZ,\ \xi\in \W.
\] 
So $\widehat V$ is a unitary on $\widehat \H$, 


For arbitrary given $e\in E^1$ and $k\in \bZ$, it follows from Lemma \ref{L:AA} that one can choose $p,q\in \bN$ and $f\in E^1$ such that 
\begin{align}
\label{E:AA}
p+k\ge 0 \text{ and  }
V^q S_e=S_fV^p.
\end{align}
Now, for each $e\in E^1$, define $\widehat S_e$ on $\widehat \H$ 
\begin{align}
\label{E:Se}
\widehat S_e(\xi\otimes e_k)=(\widehat V^q)^* (S_f V^{p+k}\xi) \qforal \xi\in \W. 
\end{align}
For each $v\in \DirectedGraph^0$, define
\begin{align}
\label{E:Sv}
\widehat S_v:=  \widehat S_e^* \widehat S_e \quad (e\in \DirectedGraph^1 v). 
\end{align}

\begin{lem}
\label{L:ExtSe}
Keep the same notation as above. Then
\begin{enumerate}
\item 
$\widehat S_e$ is well-defined and $\widehat S_e|_\H=S_e$;
\item
$\widehat V \widehat S_e=\widehat S_{1\cdot e} \widehat V_{1|_e}$ for all $e\in \DirectedGraph^1$;
\item
$\widehat S_v$ is well-defined and $\widehat S_v|_\H=S_v$;
\item
$\widehat V \widehat S_v=\widehat S_{1\cdot v} \widehat V$ for all $v\in \DirectedGraph^0$;
\end{enumerate}
\end{lem}

\begin{proof}
(i) Suppose that there are $f_i\in E^1$, $p_i, q_i\in \bN$ $(i=1,2)$ which satisfy 
\begin{align*}
p_i+k\ge 0 \text{ and  }
V^{q_i} S_e=S_{f_i}V^{p_i}\ (i=1,2).
\end{align*}
WLOG, let us assume that $q_2\ge q_1$, and so $p_2\ge p_1\ge -k$.
\[
q_j\cdot e=f_j \text{ and }q_j|_e=p_j\ge k.
\]
Since $f_1$ and $f_2$ both lie in the orbit of $e$, by Lemma \ref{L:AA}, there is $\ell\in \bN$ big enough such that 
\[
V^\ell S_{f_j}=S_{f_j} V^m\ (j=1,2)\quad \text{ and } \quad m+k\ge 0.
\]
Thus 
\begin{align*}
V^\ell S_{f_2}V^{p_2+k}
&=S_{f_2} V^{m+p_2+k}V^{q_2}S_eV^{m+k}\\
&=V^{q_2-q_1} V_{q_1} S_e V^{m+k}\\
&=V^{q_2-q_1} S_{f_1}  V^{p_1+m+k}\\
&=V^{q_2-q_1} V^\ell  S_{f_1} V^{p_1+k}\\
&=V^\ell V^{q_2-q_1}  S_{f_1} V^{p_1+k}.
\end{align*}
Therefore
\[
S_{f_2}V^{p_2+k}=V^{q_2-q_1}  S_{f_1} V^{p_1+k}.
\]
Then, for $\xi\in\W$, one has $S_{f_2}V^{p_2+k}\xi=V^{q_2-q_1} V_\ell  S_{f_1} V^{p_1+k}\xi\in\H$. This implies 
$(\widehat V^{q_2})^*S_{f_2}V^{p_2+k}\xi=(\widehat {V}^{q_1})^*  S_{f_1} V^{p_1+k}\xi$. Therefore $\widehat S_e$ in \eqref{E:Se} is well-defined. 
 
To see ${\widehat S_e}|_\H=S_e$, let $\xi\in \W$ and $k\in \bN$. Then $\widehat S_e (\xi\otimes e_k)=S_e V^k\xi$ by letting $p=q=0$ and 
$e=f$ in \eqref{E:AA}. Hence $\widehat S_e$ extends $S_e$ (through $J$). 

(ii) For $\xi\otimes e_k$, choose $p,q\in \bN$ with $p+k\ge 0$ (in particular $p+k+1|_e\ge 0)$ 
and $f\in E^1$ such that 
\begin{align}
\label{E:q1e}
V^q S_{1\cdot e}=S_f V^p.
\end{align}
Thus 
\[
\widehat S_{1\cdot e} \widehat V_{1|_e}(\xi\otimes e_k)
=\widehat S_{1\cdot e} (\xi\otimes e_{k+1|_e})
=(\widehat V^q)^*(S_ f V^{p+k+1|_e}\xi). 
\]
From
$
V^q S_{1\cdot e}=S_f V^p,
$
we have
$
V^q S_{1\cdot e}V^{1|_e}=S_f V^{p+1|_e}, 
$
implying 
\[
V^{q+1} S_e=S_f V^{p+1|_e}. 
\]
Hence 
\[
\widehat V \widehat S_e(\xi\otimes e_k)=\widehat V (\widehat V^{q+1})^* (S_f V^{p+1|_e + k})(\xi) =(\widehat V^{q})^* (S_f V^{p+1|_e + k}\xi).
\]
Therefore $\widehat V \widehat S_e=\widehat S_{1\cdot e} \widehat V_{1|_e}$ for all $e\in E^1$.

(iii) Notice that, from \eqref{E:AA} and (ii) above, one also has $\widehat V^q\widehat S_e=\widehat S_f \widehat V^{p}$, and so 
$
 \widehat S_ e (\widehat V^p)^*=(\widehat V^q)^*\widehat S_f.
$
Then
\begin{align}
\nonumber
 \widehat S_e^* \widehat S_e(\xi\otimes e_k)
&= \widehat S_e^* (\widehat V^q)^*(S_f V_{p+k}\xi)
\label{E:e*e}
= (\widehat V^p)^* \widehat S_f^* (S_f V^{p+k}\xi)\\
&= (\widehat V^p)^*  S_f^* (S_f V^{p+k}\xi) 
=(\widehat V^p)^*  S_{s(f)} V^{p+k}\xi.
\end{align}

Now, arbitrarily given $e_i\in E^1 v$ $(i=1,2)$ and $k\in \bZ$, one can choose $q_i\in \bN$ such that $q_1|_{e_1}= q_2|_{e_2}$ and $q_1|_{e_1}+k>0$.\footnote{For instance, let $q_1=N\cdot |\Omega_{e_1}|\cdot (\sum_{e\in \Omega_{e_2}}1|_{e})$, $q_2=N\cdot |\Omega_{e_2}|\cdot (\sum_{e\in \Omega_{e_1}}1|_e)$, and $N\in \bN$ big enough.}
Thus
\[
V^{q_i} S_{e_i}=S_{q_i\cdot e_i} V^{p}, \text{ where $p=q_1|_{e_1}=q_2|_{e_2}$ with $p+k\ge 0$}.  
\]
This implies 
\[
\widehat S_{e_i}(\xi\otimes e_k)=(\widehat V^{q_i})^* (S_{q_i\cdot e_i} V^{p+k}\xi) \qforal \xi\in \W. 
\] 
Notice that, from the choice $q_i$, one has $q_i\cdot v=v$ for $i=1,2$. 
Then from \eqref{E:e*e} one has 
\begin{align}
\label{E:wd}
 \widehat S_{e_i}^* \widehat S_{e_i}(\xi\otimes e_k)= (\widehat V^p)^*  S_{q_i\cdot v} V_{p+k}\xi= (\widehat V^p)^*  S_{v} V^{p+k}\xi.
 \end{align}
 Hence
 \[
  \widehat S_{e_1}^* \widehat S_{e_1}(\xi\otimes e_k)= \widehat S_{e_2}^* \widehat S_{e_2}(\xi\otimes e_k).
 \]
Therefore the definition of $\widehat S_v$ is independent of the choice of $e\in \DirectedGraph^1v$, and so $\widehat S_v$ is well-defined. 
Clearly, $\widehat S_v$ extends $S_v$. 

(iv) 
For $v\in \DirectedGraph^0$ and $e\in \DirectedGraph^1v$, on one hand, from \eqref{E:Sv} and \eqref{E:wd} we compute 
\begin{align*}
\widehat V \widehat S_v (\xi \otimes e_k)
&=\widehat V   \widehat S_e^* \widehat S_e (\xi \otimes e_k)
= \widehat V \widehat V_p^*  S_v V^{p+k}\xi\\
&= \widehat V_p^*  V S_v V^{p+k}\xi
  =  \widehat V_p^*  S_{1\cdot v} V^{p+k+1}\xi. 
\end{align*}
On the other hand, it follows from \eqref{E:Sv}, \eqref{E:q1e}, and \eqref{E:wd} that 
\[
\widehat S_{1\cdot v}\widehat V(\xi\otimes e_k)=\widehat S_{1\cdot e}^*\widehat S_{1\cdot e}\widehat V(\xi\otimes e_k) 
= (\widehat V^p)^*  S_{1\cdot v} V^{p+1+k}\xi. 
\]
Hence $\widehat V\widehat S_v= \widehat S_{1\cdot v}\widehat V$. 
\end{proof}


\begin{thm}
\label{T:Shifts}
Keep the same notation as above. Then $(\widehat \V, \widehat \S)$ is a representation of $(\bN, E)$, and is a  non-trivial dilation of $(\V, \S)$ from $\H$ to $\widehat \H$. Moreover, if $\S$ is CK, then so is $\widehat \S$. 
\end{thm}

\begin{proof}
Firstly, by Lemma \ref{L:ExtSe} (ii) and (iv), $(\widehat \V, \widehat \S)$ satisfies (SS).

Secondly, we show that $\widehat \S$ is a TCK $\DirectedGraph$-family. 

\underline{$\widehat \S$ satisfies (TCK1):}
For each $v\in \DirectedGraph^0$, to show that $\widehat S_v$ is an orthogonal projection, from \eqref{E:Sv} it suffices to verify that 
each $\widehat S_e$ is a partial isometry. For this, 
fix $\xi\otimes e_k$. 
Choose large enough $q\in \bN$ such that 
\[
V^q S_e=S_f V^{p}, \ p+k\ge 0.  
\]
So we also have $\widehat V^q\widehat S_e=\widehat S_f \widehat V^{p}$ and so 
\[
 \widehat S_ e (\widehat V^p)^*=(\widehat V^q)^*\widehat S_f.
\]
Now compute 
\begin{align*}
\widehat S_e \widehat S_e^* \widehat S_e(\xi\otimes e_k)
&=\widehat S_e \widehat S_e^* (\widehat V^q)^*(S_f V_{p+k}\xi)
=\widehat S_e (\widehat V^p)^* \widehat S_f^* (S_f V_{p+k}\xi)\\
&=\widehat S_e (\widehat V^p)^*  S_f^* (S_f V_{p+k}\xi)\ (\text{as }p+k\ge 0)\\
&=\widehat S_e (\widehat V^p)^*  S_{s(f)} V_{p+k}\xi
=(\widehat V^q)^*\widehat S_f   S_{s(f)} V_{p+k}\xi\\
&=(\widehat V^q)^* S_f   S_{s(f)} V_{p+k}\xi\ (\text{as }p+k\ge 0)\\
&=(\widehat V^q)^* S_f  V_{p+k}\xi
  =\widehat S_e(\xi\otimes e_k).
\end{align*}
Thus $\widehat S_e \widehat S_e^* \widehat S_e=\widehat S_e$. 

\underline{$\widehat \S$ satisfies (TCK2):} This is directly from the definition of $\widehat S_v$ (see \eqref{E:Sv}). 

Choose large enough $q\in \bN$ such that 
\[
V^q S_e=S_{\lambda_1} V^{p_1},\ V^q S_f=S_{\lambda_2} V^{p_2},\ p_1+k_1\ge 0,\ p_2+k_2\ge 0.  
\]
Notice that $\lambda_1=q\cdot e \ne q\cdot f=\lambda_2$ as $e\ne f$. 
Then 
\begin{align*}
\langle \widehat S_e(\xi_1\otimes e_{k_1}), \widehat S_f(\xi_2\otimes e_{k_2})\rangle 
&=\langle (\widehat V^q)^*(S_{\lambda_1} V^{p_1+k_1}\xi_1), (\widehat V^q)^*(S_{\lambda_2}V^{p_2+k_2}\xi_2)\rangle \\
&=\langle (S_{\lambda_1} V^{p_1+k_1}\xi_1), (S_{\lambda_2}V^{p_2+k_2}\xi_2)\rangle\\
&=0\ (\text{as }\lambda_1\ne \lambda_2\in E^1 \Rightarrow S_{\lambda_1}^*S_{\lambda_2}=0).
\end{align*}
Thus $\sum_{e\in v\DirectedGraph^1} \widehat S_e  \widehat S_e^*\le S_v$ for all $v\in \DirectedGraph^0$. 

Therefore $\widehat S$ is a TCK $\DirectedGraph$-family. 

Thirdly, since $\widehat V$ is unitary, $(\widehat \V, \widehat \S)$ is Nica-covariant by Proposition~\ref{prop.automatic.NC}. 

Finally, if $\S$ is CK, then so is $\widehat \S$. To the contrary, suppose that $\sum \widehat S_e\widehat S_e^*<I$. 
Since $\{\xi\otimes e_k: \xi\in \W, k\in \bZ\}$ is dense in $\widehat \H$,
there is $h\in \spn\{\xi\otimes e_k: \xi\in \W, k\in \bZ\}$ such that $h\perp \widehat S_e\widehat H$ for all $e\in E^1$. 
Let $m\in \bN$ be large enough so that 
\[
\widehat V_m h\in \H,\ \widehat V_m \widehat S_e=\widehat S_e V_{m_e}
\]
for some $m_e\in \bN$. Since $(\widehat V, \widehat S)$ is Nica-covariant, one has 
\[
 \widehat V_m^* \widehat S_e=\widehat S_e V_{m_e}^*.
\] 
Thus for all $x\in \H$
\begin{align*}
\langle \widehat V_m h, \widehat S_e x\rangle 
= \langle  h, \widehat V_m^*\widehat S_e x\rangle = \langle h, \widehat S_e\widehat V_{m_e}^* x\rangle 
= 0 \ (\text{as } \widehat S_e\widehat V_{m_e}^* x\in \widehat S_e\widehat \H).
\end{align*}
Thus $\widehat V_m h\perp \widehat S_e \widehat \H$ for all $e\in E^1$. 
But $\widehat V_m h \in \H$ and $\sum_{e\in E^1} \widehat S_e \widehat S_e^*|_\H=I_\H$ as $\H$ reduces $S$. So $\widehat V_m h=0$ and thus $h=0$. 
\end{proof}

As an immediate consequence of Theorem \ref{T:Shifts}, we obtain the following. 

\begin{cor}\label{cor.shift.not.max}
If $(\V,\S)$ is a representation of $(\bN, E)$ which is of pure + CK type or left regular, then there is a non-trivial dilation. 
\end{cor}

We want to point out that when $(\V,\S)$ is atomic, the dilation obtained in Theorem~\ref{T:Shifts} is also atomic in a natural way. This is in sharp contrast with the unitary + pure shift case discussed in the next subsection.

\begin{prop}\label{prop.atomic.dilation} Let $(\V,\S)$ be an atomic representation of pure + CK type on a Hilbert space $\H$ with atomic vectors $\{\xi_i\}_{i\in I}$. Then its dilation $(\widehat{\V}, \widehat{\S})$ constructed in Theorem~\ref{T:Shifts} on the Hilbert space $\widehat{\H}$ is also atomic where $\{\xi_i\}_{i\in I}$ are still atomic vectors. 
\end{prop}

\begin{proof} Pick an orthonormal basis $\{k_i\}_{i\in I}$ for the atomic representation $(\V,\S)$ such that each $Vk_i, S_e k_i$ equals to some $\omega k_j$, $|\omega|=1$. The space $\W=\ker V^*$ has an orthogonal basis $\{k_j\}_{j\in J}$ where $k_j\notin \ran V$. By re-scaling each $k_i$ with an appropriate unit scalar, we may assume $\{k_i\}_{i\in I}=\{V^m k_j: j\in J, m\geq 0\}$. 

The space $\widehat{\H}$ in Theorem~\ref{T:Shifts} is given by $\W\otimes \ell^2(\mathbb{Z})$, which has an orthonormal basis $\{\widehat{V}^m k_j: j\in J, m\in\mathbb{Z}\}$. Notice for $m\geq 0$, $\widehat{V}^m k_j = V^m k_j$, so that each atomic vector $\{k_i\}_{i\in I}$ still belongs to the new orthonormal basis. 

Now fix this orthonormal basis $\{\widehat{V}^m k_j: j\in J, m\in\mathbb{Z}\}$. Clearly $\widehat{V}$ maps a basic vector to another basic vector. For each basic vector $\widehat{V}^m k_j$ and $e\in E^1$, choose $p,q\in\mathbb{N}$ and $f\in E^1$ such that $p+m\geq 0$ and $V^q S_e = S_f V^p$. We have:
\[
\widehat S_e(\widehat{V}^m k_j)=(\widehat V^q)^* (S_f V^{p+m} k_j)
\]
Here, since $(\V,\S)$ is atomic, $S_f V^{p+m} k_j$ is one of $k_i$, $i\in I$, we have $k_i=V^{m'} k_{j'}$ for some $m'\geq 0$ and $j'\in J$. Therefore, 
\[
\widehat S_e(\widehat{V}^m k_j)=(\widehat V^q)^* (S_f V^{p+m} k_j) = \widehat{V}^{m'-q} k_{j'}.
\]
Hence, $(\widehat{\V},\widehat{\S})$ is also atomic, with respect to the orthonormal basis $\{\widehat{V}^m k_j: j\in J, m\in\mathbb{Z}\}$.
\end{proof}

\subsection{The unitary + pure shift type}
This is Case (iii) in Theorem~\ref{thm.Wold.main}. 
Proposition~\ref{prop.atomic.dilation} shows that for an atomic pure + CK type representation, one can extend the atomic vectors to an atomic dilation. In contrast, we first demonstrate that this is not the case for unitary + pure shift type: one cannot simply extend the atomic vectors of an atomic unitary + pure shift type representation to an atomic dilation.  

\begin{eg} Consider the odometer self-similar graph $\bO_n$, 
and the atomic representation $c^1=(\U,\S)$ described in Section~\ref{sec.atomic.pureE}. Here, $\H=\ell^2(\mathbb{F}_n^+)$, and we have $U\xi_{\mu} = \xi_{1\cdot \mu}$ and $S_{e_i} \xi_\mu = \xi_{e_i \mu}$. We claim that there is no nontrivial atomic dilation of $c^1$ where $\{\xi_\mu: \mu\in\mathbb{F}_n^+\}$ are still the atom vectors. 

We first notice that since $U$ is unitary, any dilation $(\widehat{\U}, \widehat{\S})$ on $\widehat{\H}\supset \H$ must have $\H$ as a reducing subspace for $\widehat{\U}$. Therefore, if the dilation is non-trivial, $\H$ cannot reduce $\widehat{S}$. This implies that some $\widehat{S}_i$ must map an atomic vector $\xi_{-1}$ outside $\H$ into an atomic vector in $\H$. 

Since each atom vector $\xi_{\mu}$ with $|\mu|>0$ is in some $S_{e_i}\H$, only the wandering vector $\xi_e$ is not in the range of any $S_{e_i}$. Therefore, we must have an atom vector $\xi_{-1}$ and edge $e_i$ such that $\widehat{S}_{e_i} \xi_{-1} = \xi_e$. We now claim that this is impossible. 

Suppose $i<n$, then 
\[\widehat{U} \widehat{S}_{e_i} \xi_{-1} = \widehat{S}_{e_{i+1}} \xi_{-1}.\]
On the other hand, 
\[\widehat{U} \widehat{S}_{e_i} \xi_{-1} = \widehat{U} \xi_e = \xi_e.\]
This implies $\xi_e$ is in the range of both $\widehat{S}_{e_i}$ and $\widehat{S}_{e_{i+1}}$, a contradiction. 

Suppose $i=n$, then
\[\widehat{U} \widehat{S}_{e_n} \xi_{-1} = \widehat{U}\xi_e = \xi_e.\]
But the Toeplitz condition tells us 
\[\widehat{U} \widehat{S}_{e_n} \xi_{-1}=\widehat{S}_{e_1} \widehat{U}\xi_{-1}.\]
This again leads to $\xi_e$ to be in the range of both $\widehat{S}_{e_n}$ and $\widehat{S}_{e_1}$, a contradiction. 
\end{eg}

In light of this example, building dilations for unitary + pure shift type representations is much more complicated. We shall first prepare some notation. 

From the Wold decomposition (Theorem~\ref{thm.Wold.main}), for each equivalence class $\Omega_v=\{n\cdot v: n\in\mathbb{N}\}$, there could be a unitary + pure shift type component $(\U, \S)$, where $\U$ is unitary and $\S$ is a direct sum of some amplifications of $\oplus_{u\in\Omega_v} L_{E,u}$. Let us fix such a representation $(\U,\S)$ on a Hilbert space $\H$ for the orbit $\Omega_v$. 

Let us relabel the vertices in $\Omega_v$ as $\{v_0, v_1, \dots, v_{m-1}\}$ where $1\cdot v_i=v_{(i+1)\!\!\mod m}$.
For each $v_i\in \Omega_v$, define $\W_{v_i}=(S_{v_i}-\sum_{e\in v_i \DirectedGraph^1} S_e S_e^*)\H$ and $\H_{v_i}=\oplus_{\mu\in E^* v_i} S_\mu \W_{v_i}$. We have $\H=\oplus_{i=0}^{m-1} \H_{v_i}$. Also, denote $\W=\oplus_{i=0}^{m-1} \W_{v_i}$ 

Pick an arbitrary edge, WLOG, say $e_0$, with $s(e_0)=v_0$. Let its orbit under the $\mathbb{N}$-action be $\Omega_{e_0}=\{e_0, e_1, \dots, e_{q-1}\}$.
Here, $1\cdot e_j=e_{(j+1)\!\! \mod q}$. 
One may also notice that since $e_j=j\cdot e_0$, we have $s(e_j) = j\cdot s(e_0)=j\cdot v_0 = v_{j\!\! \mod m}$. In particular, we must have $m\mid q$. 

Now notice from Remark \ref{R:1|e m} that $1|_{e_j} = 1 \!\! \mod m$. 

Set
\[
M:=\sum_{j=0}^{q-1} 1|_{e_j}.
\]
 Alternatively, $M$ can also be computed as
\[M= \sum_{j=0}^{q-1} 1|_{e_j}= \sum_{j=0}^{q-1} 1|_{j\cdot e_0}=q|_{e_0}.\]
By Assumption~\ref{item:A1}, at least one of $1|_{e_i}$ is strictly positive. By rearranging $v_i$ and $e_j$, we may assume without loss of generality that $1|_{e_{q-1}}\neq 0$. In particular, $M>0$.

Define 
\[\widehat{\H} =\H\otimes \mathbb{C}^M.\]
Take the standard orthonormal basis $\{w_j: 0\leq j\leq M-1\}$ on $\mathbb{C}^M$, we denote $\H_j = \H\otimes w_j \cong \H$, and thus $\widehat{\H}=\oplus_{j=0}^{M-1} \H_j$. 

For each $\mu\in E^*$, define $\widehat{S}_\mu (\xi\otimes w_j)=(S_\mu\xi) \otimes w_j$ (in other words, $\widehat{S}_\mu = S_\mu \otimes I$ on $\widehat{H}=\H\otimes \mathbb{C}^M$). Then it is clear that $\widehat{\S}$ is a pure shifty type TCK $E$-family. 

Now each $\H_j=\oplus_{\mu\in E^*} S_\mu \W_j$, where $\W_j = \W\otimes w_j \cong \W$. Let $\widehat\W=\oplus \W_j$, which is also the wandering space for $\widehat{\S}$. Let us define a unitary $\widehat{U}_0$ on $\widehat\W$ by the following formula:
\[
\widehat{U}_0 (\xi\otimes w_j) = \begin{cases}
    (U\xi) \otimes w_{j+1}, \quad & \text{ if } 0\leq j<M-1, \\
    (U^{*(M-q-1)}\xi) \otimes w_0, \quad & \text{ if } j=M-1.
\end{cases}
\]
\begin{lem} For each $0\leq i\leq m-1$, $\widehat{U}_0$ is a unitary from $\widehat{S}_{v_i}\widehat\W$ to $\widehat{S}_{v_{i+1}} \widehat\W$.    
\end{lem}

\begin{proof} For each $v_i$, $\widehat{S}_{v_i} \widehat{\W} = (S_{v_i} \W) \otimes \mathbb{C}^M$, which contains all vectors of the form $\xi\otimes w_j$, $\xi\in S_{v_i}\W$.  

Take any $\xi\otimes w_j\in \widehat{S}_{v_i} \widehat\W$ where $\xi\in S_{v_i}\W$.  If $0\leq j<M-1$, we have 
\[\widehat{U}_0 (\xi\otimes w_j) = (U\xi) \otimes w_{j+1},\]
where $U\xi\in S_{v_{i+1}} \W$ and thus $\widehat{U}_0 (\xi\otimes w_j) \in \widehat{S}_{v_{i+1}} \widehat\W$.

When $j=M-1$, 
\[\widehat{U}_0 (\xi\otimes w_j) =(U^{*(M-q-1)}\xi) \otimes w_0.\]
Notice for $\xi \in S_{v_i} \W$, $U^k \xi \in S_{v_{i+k \!\!\mod m}} \W$. 
Notice $1_{e_j} = 1 \!\!\mod m$ and thus 
\[M = \sum_{j=0}^{q-1} 1|_{e_j} = q \!\!\mod m.\]
Thus $-(M-q-1)=1 \!\!\mod m$ and we have $U^{*(M-q-1)}\xi\in  S_{v_{i+1}} \W$ and thus $\widehat{U}_0 (\xi\otimes w_j) \in \widehat{S}_{v_{i+1}} \widehat\W$ in this case as well. 
\end{proof}

Now, by Lemma~\ref{lemma:key_lemma}, there is a unique unitary $\widehat{U}$ on $\widehat\H$ such that $(\widehat\U, \widehat\S)$ is a unitary + pure shift type representation and $\widehat\U |_{\widehat\W} = \widehat{U}_0$. Notice from the proof of Lemma~\ref{lemma:key_lemma}, we can write down the unitary $\widehat{U}$ explicitly. Since $\widehat\S$ is a pure shift type $E$-family, $\widehat\H=\bigoplus_{\mu\in E^*} \widehat{S}_\mu \W$. The map $\widehat{U}$ is defined by
\[\widehat{U} (\widehat{S}_\mu\widehat{\xi}) = \widehat{S}_{1\cdot \mu} \widehat{U}_0^{1|_\mu}\widehat\xi, \quad \widehat{\xi}\in\widehat\W.\]
We now embed $\H$ inside $\widehat\H$.

\begin{lem} Let $J: \W \to \widehat{\H}$ be defined by
\[J\xi = \frac{1}{\sqrt{q/m}} \sum_{j=0}^{q-1}  \left(S_{e_j} U^{j|_{e_0}-j} \xi\right)\otimes w_{j|_{e_0}}.\]
Then $J$ is an isometry. Moreover, this can be extended to an isometric embedding $\widehat{J}:\H\to \widehat{\H}$ by
\[\widehat{J}(S_\mu \xi) = \widehat S_\mu (J\xi), \quad \xi\in \W.\]
\end{lem}

\begin{proof} First of all, on $\H$, we have $\sum_{i=0}^{m-1} S_{v_i}=I$. Take any $\xi\in\W$, define $\xi_i=S_{v_i} \xi$. We have $\xi=\sum_{i=0}^{m-1} \xi_i$ and $\xi_i\in S_{v_i}\W$. Moreover, since $\xi_i$ are orthogonal to each other, 
\[\|\xi\|^2=\sum_{i=0}^{m-1} \|\xi_i\|^2.\]

For each $0\leq j\leq q-1$, $j_{e_0}-j = 0 \!\!\mod m$. Therefore, $U^{j|_{e_0}-j} \xi_i \in S_{v_i}\W$ and since $s(e_j)=v_{j\!\!\mod m}$, we have,
\[S_{e_j} U^{j|_{e_0}-j} \xi = \sum_{i=0}^{m-1} S_{e_j} U^{j|_{e_0}-j} \xi_i = S_j U^{j|_{e_0}-j} \xi_{j \operatorname{mod} m}.\]
Hence, we can write
\[J\xi = \frac{1}{\sqrt{q/m}} \sum_{j=0}^{q-1}  \left(S_{e_j} U^{j|_{e_0}-j} \xi_{j \operatorname{mod} m}\right)\otimes w_{j|_{e_0}}.\]
Since $S_{e_j}$ has orthogonal ranges, we have the vectors $\left\{\left(S_{e_j} U^{j|_{e_0}-j} \xi_{j \operatorname{mod} m}\right)\otimes w_{j|_{e_0}}\right\}$ are orthogonal to one another. Since $U$ is unitary and $S_{e_j}$ is isometric on $S_{v_{j\!\!\mod m}} \W$, we have 
\[\left\|\left(S_{e_j} U^{j|_{e_0}-j} \xi_{j \operatorname{mod} m}\right)\otimes w_{j|_{e_0}}\right\| = \|\xi_j\|.\]
Therefore,
\[\|J\xi\|^2 = \frac{1}{q/m} \sum_{j=0}^{q-1} \|\xi_{j \operatorname{mod} m}\|^2=\|\xi\|^2.\]
This proves that $J$ is an isometry on $\W$. Now define  $\widehat{J}:\H=\bigoplus_{\mu\in E^*} S_\mu\W \to \widehat{\H}$ by
\[\widehat{J}(S_\mu \xi) = \widehat S_\mu (J\xi), \quad \xi\in \W.\]
The map $\widehat{J}$ is clearly isometric from $S_\mu \W$ to $\widehat{S}_\mu (J\W)$. The space $\widehat{S}_\mu (J\W)$ consists of vectors of the form 
\[\frac{1}{\sqrt{q/m}} \sum_{j=0}^{q-1}  \left(S_{\mu}S_{e_j} U^{j|_{e_0}-j} \xi\right)\otimes w_{j|_{e_0}},\quad \xi\in\W.\]
For $\mu\neq\nu \in E^*$, the range of $S_\mu S_{e_j}$ is always orthogonal to that of $S_\nu S_{e_j}$. Therefore, the spaces $\widehat{S}_\mu (J\W)$ are pairwise orthogonal. Hence, $\widehat{J}$ is an isometry.
\end{proof}

We now claim that $(\widehat{\U},\widehat{\S})$ is a dilation of $(\U,\S)$ via $\widehat{J}$.

\begin{prop}\label{prop.unitary.pure} We have $U=\widehat{J}^* \widehat{U} \widehat{J}$ and $\S_\mu=\widehat{J}^* \widehat S_\mu \widehat{J}$ for all $\mu\in E^*$. Moreover, $(\U,\S)$ is not maximal. 
\end{prop}

\begin{proof} First, for any $\mu\in E^1$, $\nu\in E^*$, and $\xi\in\W$, we have:
\[
    \widehat{S}_\mu \widehat{J} (S_\nu \xi) 
    = \widehat{S}_\mu \widehat{S}_\nu J\xi 
    = \widehat{J} S_{\mu\nu} \xi 
    = \widehat{J} S_\mu (S_\nu\xi).
\]
Therefore, $\widehat{J}^* \widehat{S}_\mu \widehat{J} = S_\mu$. 

Next, we first verify that $U=\widehat{J}^* \widehat{U} \widehat{J}$ on $\W$. Pick any $\xi\in \W$, we have
\begin{align*}
    \widehat{U} \widehat{J} \xi &= \widehat{U}  \left(\frac{1}{\sqrt{q/m}} \sum_{j=0}^{q-1}  \left(S_{e_j} U^{j|_{e_0}-j} \xi\right)\otimes w_{j|_{e_0}}\right) \\
    &= \frac{1}{\sqrt{q/m}} \left(\sum_{j=0}^{q-1}  \widehat{S}_{1\cdot e_j} \widehat{U}_0^{1|_{e_j}} \left(U^{j|_{e_0}-j} \xi\otimes w_{j|_{e_0}}\right)\right).
\end{align*}
For $j<q-1$, $1|_{e_j}+j|_{e_0}=(j+1)|_{e_0}\leq M-1$ since we assumed $1|_{e_{q-1}}>0$. Therefore, by the choice of $\widehat{U}_0$, we have:
\[\widehat{U}_0^{1|_{e_j}} \left(U^{j|_{e_0}-j} \xi\otimes w_{j|_{e_0}}\right) = 
U^{j|_{e_0}-j+1|_{e_j}} \xi\otimes w_{j|_{e_0}+1|_{e_j}}
= U^{(j+1)|_{e_0}-j} \xi\otimes w_{(j+1)|_{e_0}}.\]
For $j=q-1$, 
\begin{align*}
    \widehat{U}_0^{1|_{e_{q-1}}} \left(U^{(q-1)|_{e_0}-(q-1)} \xi\otimes w_{(q-1)|_{e_0}}\right) 
    &= 
\widehat{U}_0 \left(U^{(q-1)|_{e_0}-(q-1)+(1|_{e_{q-1}}-1)} \xi\otimes w_{(q-1)|_{e_0}+(1|_{e_{q-1}}-1)}\right) \\
&=
\widehat{U}_0 \left(U^{q|_{e_0}-q} \xi\otimes w_{q|_{e_0}-1)}\right) \\
&= \widehat{U}_0 \left(U^{M-q} \xi\otimes w_{M-1)}\right)\\
&= U^{*(M-q-1)} U^{M-q}\xi \otimes w_0 \\
&= U\xi \otimes w_0.
\end{align*}
Therefore,
\begin{align*}
    \widehat{U} \widehat{J} \xi &= \frac{1}{\sqrt{q/m}}
\left(U\xi\otimes w_0 + \sum_{j=0}^{q-2} U^{(j+1)|_{e_0}-j} \xi\otimes w_{(j+1)|_{e_0}} \right) \\
&= \frac{1}{\sqrt{q/m}}
\left(U^{0|_{e_0}-0} U\xi\otimes w_{0|_{e_0}} + \sum_{j'=1}^{q-1} U^{j'|_{e_0}-j'+1} \xi\otimes w_{j'|_{e_0}} \right) \\
&= \frac{1}{\sqrt{q/m}}\left(\sum_{j'=0}^{q-1} U^{j'|_{e_0}-j'} (U\xi)\otimes w_{j'|_{e_0}} \right) \\
&= \widehat{J} (U\xi).
\end{align*}

Finally, since $\widehat{J}^* (\widehat{\U}, \widehat{\S}) \widehat{J}$ is a unitary + pure shift type representation, where $\widehat{J}^* \widehat{\S} \widehat{J}=\S$ and $\widehat{J}^* \widehat{\U}\widehat{J}$ coincides with $\U$ on the wandering space $\W$, Lemma~\ref{lemma:key_lemma} establishes the uniqueness of such representations and thus $\widehat{J}^* \widehat{\U}\widehat{J} = \U$. 
\end{proof}

\begin{eg} Let us consider the following graph $E$:
\begin{figure}[th]
    \centering

    \begin{tikzpicture}[scale=1, every node/.style={scale=1}]
        
        \draw[-Stealth] plot [smooth]  coordinates 
        {(0,0) (0.3,0.1) (1, 0.2) (1.7, 0.1) (2,0)};
        \draw[-Stealth] plot [smooth]  coordinates {(2,0) (1.7,-0.1) (1, -0.2) (0.3, -0.1) (0,0)};

        \draw[-Stealth] plot [smooth]  coordinates 
        {(0,0) (-0.4,0.2) (-0.5,0) (-0.4,-0.2) (0,0)};
        \draw[-Stealth] plot [smooth]  coordinates 
        {(2,0) (2+0.4,0.2) (2+0.5,0) (2+0.4,-0.2) (2,0)};
        \node at (0,0) {$\bullet$};
        \node at (2,0) {$\bullet$};
        \node at (-0.7,0) {$e_0$};
        \node at (2.7,0) {$e_1$};
        \node at (1, 0.4) {$f_0$};
        \node at (1, -0.4) {$f_1$};
        \node at (0, 0.25) {$v_0$};
        \node at (2, 0.25) {$v_1$};
    \end{tikzpicture}
\end{figure}

Define a self-similar $\mathbb{N}$-action on $E$ by
\[1\cdot v_i=v_{1-i},\quad 1\cdot e_i = e_{1-i},\quad 1\cdot f_i=f_{1-i},\]
and
\[1|_{e_0}=3,\quad 1|_{e_1}=1,\quad 1|_{f_0}=3,\quad 1|_{f_1}=3.\]
Then we obtain a self-similar graph $(\bN, \DirectedGraph)$. 

Let $\S$ be the left-regular representation of $E$. Let $\delta_{v_0}$ and $\delta_{v_1}$ be the two wandering vectors for $S_{v_0}$ and $S_{v_1}$ respectively. 
Then $\W=\spn\{\delta_0,\delta_1\}$ is a wandering space for $\S$. Define a unitary $U_0$ on $\W$ by $U_0\delta_{v_0}=\delta_{v_1}$ 
and $U_0\delta_{v_1}=\lambda\delta_{v_0}$ for some $\lambda\in\mathbb{T}$. By Lemma~\ref{lemma:key_lemma}, there is a unique unitary + pure shift type representation $(\U,\S)$ such that $U|_\W=U_0$. This $U$ is defined uniquely using the formula
\[U(S_\mu \delta_{v_i}) = S_{1\cdot \mu} U_0^{1|_\mu} \delta_{v_i}\quad \text{for } \mu\in E^* \text{ with } s(\mu)=v_i.\]

This is an atomic representation on $\H=\overline{\spn}\{\delta_\mu: \mu\in E^*\}$ with wandering space $\W=\{\delta_{v_0}, \delta_{v_1}\}$. Pick the orbit $\Omega_{e_0}=\{e_0, e_1\}$. In this example, we have $q=m=2$ and $M=2|_{e_0}=4$. The space $\widehat\H$ from Proposition~\ref{prop.unitary.pure} is defined as $\widehat\H=\H\otimes \mathbb{C}^4$ with orthonormal basis $\{\delta_\mu\otimes w_j: \mu\in E^*, 0\leq j\leq 3\}$. The pure shift TCK $\DirectedGraph$-family $\widehat\S$ is essentially $\S\otimes I$ on $\widehat{\H}$. The map $\widehat{U}_0$ is defined on the wandering space $\widehat\W=\spn\{\delta_{v_i}\otimes w_j: 0\leq i\leq 1, 0\leq j\leq 3\}$, given by the formula:
\begin{align*}
    \widehat{U}_0 \delta_{v_i}\otimes w_j &= \begin{cases}
        (U\delta_{v_i}) \otimes w_{j+1}, & \text{ if } 0\leq j\leq 2, \\
        (U^*\delta_{v_i}) \otimes w_0, & \text{ if } j=3, \\
    \end{cases} \\
    &=
    \begin{cases}
        \delta_{v_1} \otimes w_{j+1}, & \text{ if } 0\leq j\leq 2, i=0, \\
        \lambda \delta_{v_0} \otimes w_{j+1}, & \text{ if } 0\leq j\leq 2, i=1, \\
        \delta_{v_{1}} \otimes w_0, & \text{ if } j=3, i=0, \\
        \overline{\lambda}\delta_{v_{0}} \otimes w_0, & \text{ if } j=3, i=1. \\
    \end{cases} \\
\end{align*}
This is extended to a unique unitary on $\widehat{\H}$ such that $(\U,\S)$ is a unitary + pure shift Toeplitz representation of $(\bN, \DirectedGraph)$. 

The map $J$ that embeds in $\W$ inside $\widehat\H$ can be described by
\begin{align*}
    J\delta_{v_0} &= S_{e_0} \delta_{v_0} \otimes w_0, 
 = \delta_{e_0} \otimes w_0 \\
    J\delta_{v_1} &= S_{e_1} U^2 \delta_{v_1} \otimes w_3
= \lambda \delta_{e_1} \otimes w_3.
\end{align*}
It is easy to see that $J$ maps an orthonormal basis of $\W$ to an orthonormal set in $\widehat\H$ and thus it is an isometry. The map $\widehat{J}$ that embeds $\H$ inside $\widehat\H$ is defined by
\begin{align*}
    \widehat{J}\delta_\mu &= \widehat{S}_{\mu} J \delta_{s(\mu)} \otimes w_0 
    = \begin{cases}
       \delta_{\mu e_0} \otimes w_0, & \text{ if } s(\mu)=v_0, \\
       \lambda \delta_{\mu e_1} \otimes w_3, & \text{ if } s(\mu)=v_1.
    \end{cases}
\end{align*}

To see $\widehat{J}^* (\widehat\U,\widehat\S)\widehat{J}$ is indeed $(\U,\S)$, it suffices to verify that they agree on the wandering space $\W$. For this, we can compute
\begin{align*}
     \widehat{U} \widehat{J} \delta_{v_0} 
     &= \widehat{U} \widehat{S}_{e_0} \delta_{v_0}\otimes w_0 
     = \widehat{S}_{e_1} \widehat{U}_0^3 \delta_{v_0}\otimes w_0
     = \lambda \delta_{e_1} \otimes w_3 = \widehat{J} U\delta_{v_0},\\
     \widehat{U} \widehat{J} \delta_{v_1} 
     &= \widehat{U} \lambda \widehat{S}_{e_1} \delta_{v_1} 1\otimes w_3 
        = \widehat{S}_{e_0} \widehat{U}_0 \lambda \delta_{v_1}\otimes w_3 
        = \delta_{e_0} \otimes w_0 = \widehat{J} U\delta_{v_0}.
\end{align*}

We would like to point out that each orbit of edges produces a potentially different dilation. If we pick the orbit $\Omega_{f_0}$ instead, we can build another non-trivial dilation on $\H\otimes \mathbb{C}^6$. 
\end{eg}

We can now prove Theorem~\ref{thm.main.envelope}. 

\begin{proof}[\rm \textbf{Proof of Theorem~\ref{thm.main.envelope}}] As a result of Corollary~\ref{cor.shift.not.max} and Proposition~\ref{prop.unitary.pure},  we have now proven that among four cases in the Wold-decomposition Theorem~\ref{thm.Wold.main}, only the unitary + CK type representation is maximal. Therefore, the C*-envelope of $\A_{\mathbb{N},E}$ is the universal C*-algebra generated by unitary + CK type representation. Moreover, a unitary representation $\U$ of $\mathbb{N}$ coincides with a unitary representation of $\mathbb{Z}$. Hence, the universal C*-algebra generated by unitary + CK type representations is precisely the self-similar graph C*-algebra 
$\O_{\mathbb{Z},E}$.
\end{proof}


\begin{thebibliography}{10}

\bibitem{ABCD2021}
A.~an~Huef, B.~Nucinkis, C.~F. Sehnem, and D.~Yang.
\newblock Nuclearity of semigroup {$\rm C^\ast$}-algebras.
\newblock {\em J. Funct. Anal.}, 280(2):Paper No. 108793, 46, 2021.

\bibitem{ArvesonSubalgI}
W.~B. Arveson.
\newblock Subalgebras of {$C\sp{\ast} $}-algebras.
\newblock {\em Acta Math.}, 123:141--224, 1969.

\bibitem{BGN2003}
L.~Bartholdi, R.~Grigorchuk, and V.~Nekrashevych.
\newblock From fractal groups to fractal sets.
\newblock In {\em Fractals in {G}raz 2001}, Trends Math., pages 25--118.
  Birkh\"{a}user, Basel, 2003.

\bibitem{BKQS}
E.~B\'{e}dos, S.~Kaliszewski, J.~Quigg, and J.~Spielberg.
\newblock On finitely aligned left cancellative small categories,
  {Z}appa-{S}z\'{e}p products and {E}xel-{P}ardo algebras.
\newblock {\em Theory Appl. Categ.}, 33:Paper No. 42, 1346--1406, 2018.

\bibitem{BBD2023}
K.~A. Brix, C.~Bruce, and A.~Dor-On.
\newblock Normal coactions extend to the {C}*-envelope.
\newblock {\em https://arxiv.org/abs/2309.04817}, 2023.

\bibitem{BPRRW2017}
N.~Brownlowe, D.~Pask, J.~Ramagge, D.~Robertson, and M.~F. Whittaker.
\newblock Zappa-{S}z\'{e}p product groupoids and {$C^*$}-blends.
\newblock {\em Semigroup Forum}, 94(3):500--519, 2017.

\bibitem{BRRW}
N.~Brownlowe, J.~Ramagge, D.~Robertson, and M.~F. Whittaker.
\newblock Zappa-{S}z\'ep products of semigroups and their {$C^\ast$}-algebras.
\newblock {\em J. Funct. Anal.}, 266(6):3937--3967, 2014.

\bibitem{CaHR2016}
L.~O. Clark, A.~an~Huef, and I.~Raeburn.
\newblock Phase transitions on the {T}oeplitz algebras of {B}aumslag-{S}olitar
  semigroups.
\newblock {\em Indiana Univ. Math. J.}, 65(6):2137--2173, 2016.

\bibitem{DDL2019}
K.~R. Davidson, A.~Dor-On, and B.~Li.
\newblock Structure of free semigroupoid algebras.
\newblock {\em J. Funct. Anal.}, 277(9):3283--3350, 2019.

\bibitem{DFK2014}
K.~R. Davidson, A.~H. Fuller, and E.~T.~A. Kakariadis.
\newblock Semicrossed products of operator algebras by semigroups.
\newblock {\em Mem. Amer. Math. Soc.}, 247(1168):v+97, 2017.

\bibitem{DK2011}
K.~R. Davidson and E.~G. Katsoulis.
\newblock Operator algebras for multivariable dynamics.
\newblock {\em Mem. Amer. Math. Soc.}, 209(982):viii+53, 2011.

\bibitem{DP1999}
K.~R. Davidson and D.~R. Pitts.
\newblock Invariant subspaces and hyper-reflexivity for free semigroup
  algebras.
\newblock {\em Proc. London Math. Soc. (3)}, 78(2):401--430, 1999.

\bibitem{DKKLL2022}
A.~Dor-On, E.~T.~A. Kakariadis, E.~Katsoulis, M.~Laca, and X.~Li.
\newblock C*-envelopes for operator algebras with a coaction and co-universal
  {C}*-algebras for product systems.
\newblock {\em Adv. Math.}, 400:Paper No. 108286, 40, 2022.

\bibitem{DM2017}
A.~Dor-On and D.~Markiewicz.
\newblock {$\rm C^*$}-envelopes of tensor algebras arising from stochastic
  matrices.
\newblock {\em Integral Equations Operator Theory}, 88(2):185--227, 2017.

\bibitem{DM2005}
M.~A. Dritschel and S.~A. McCullough.
\newblock Boundary representations for families of representations of operator
  algebras and spaces.
\newblock {\em J. Operator Theory}, 53(1):159--167, 2005.

\bibitem{DL2022}
A.~Duwenig and B.~Li.
\newblock The {Z}appa-{S}z\'{e}p product of a {F}ell bundle and a groupoid.
\newblock {\em J. Funct. Anal.}, 282(1):Paper No. 109268, 42, 2022.

\bibitem{EP2017}
R.~Exel and E.~Pardo.
\newblock Self-similar graphs, a unified treatment of {K}atsura and
  {N}ekrashevych {$\rm C^*$}-algebras.
\newblock {\em Adv. Math.}, 306:1046--1129, 2017.

\bibitem{Grigorchuk1984}
R.~I. Grigorchuk.
\newblock Degrees of growth of finitely generated groups and the theory of
  invariant means.
\newblock {\em Izv. Akad. Nauk SSSR Ser. Mat.}, 48(5):939--985, 1984.

\bibitem{Grigorchuk1980}
R.~I. Grigor\v{c}uk.
\newblock On {B}urnside's problem on periodic groups.
\newblock {\em Funktsional. Anal. i Prilozhen.}, 14(1):53--54, 1980.

\bibitem{Humeniuk2020}
A.~Humeniuk.
\newblock C*-envelopes of semicrossed products by lattice ordered abelian
  semigroups.
\newblock {\em J. Funct. Anal.}, 279(9):108731, 46, 2020.

\bibitem{JK2005}
M.~T. Jury and D.~W. Kribs.
\newblock Ideal structure in free semigroupoid algebras from directed graphs.
\newblock {\em J. Operator Theory}, 53(2):273--302, 2005.

\bibitem{BLi2021}
B.~Li.
\newblock {$\rm C^*$}-envelope and dilation theory of semigroup dynamical
  systems.
\newblock {\em Integral Equations Operator Theory}, 93(3):Paper No. 19, 29,
  2021.

\bibitem{BLi2022}
B.~Li.
\newblock Wold decomposition on odometer semigroups.
\newblock {\em Proc. Roy. Soc. Edinburgh Sect. A}, 152(3):738--755, 2022.

\bibitem{LY2022}
B.~Li and D.~Yang.
\newblock Zappa-{S}z\'{e}p actions of groups on product systems.
\newblock {\em J. Operator Theory}, 88(2):247--274, 2022.

\bibitem{LY2019}
H.~Li and D.~Yang.
\newblock K{MS} states of self-similar {$k$}-graph {$\rm C^*$}-algebras.
\newblock {\em J. Funct. Anal.}, 276(12):3795--3831, 2019.

\bibitem{LY-ETDS}
H.~Li and D.~Yang.
\newblock The ideal structures of self-similar {$k$}-graph {C}*-algebras.
\newblock {\em Ergodic Theory Dynam. Systems}, 41(8):2480--2507, 2021.

\bibitem{LY2019-IMRN}
H.~Li and D.~Yang.
\newblock Self-similar {$k$}-graph {${\rm C}^*$}-algebras.
\newblock {\em Int. Math. Res. Not. IMRN}, (15):11270--11305, 2021.

\bibitem{XLiSmallCategory}
X.~Li.
\newblock Left regular representations of {G}arside categories {II}. finiteness
  properties of topological full groups.
\newblock {\em https://arxiv.org/pdf/2110.04505}, 2021.

\bibitem{MS1998}
P.~S. Muhly and B.~Solel.
\newblock Tensor algebras over {$C^*$}-correspondences: representations,
  dilations, and {$C^*$}-envelopes.
\newblock {\em J. Funct. Anal.}, 158(2):389--457, 1998.

\bibitem{Nekrashevych2006}
V.~Nekrashevych.
\newblock Self-similar inverse semigroups and {S}male spaces.
\newblock {\em Internat. J. Algebra Comput.}, 16(5):849--874, 2006.

\bibitem{Nekrashevych2009}
V.~Nekrashevych.
\newblock {$C^*$}-algebras and self-similar groups.
\newblock {\em J. Reine Angew. Math.}, 630:59--123, 2009.

\bibitem{Nica1992}
A.~Nica.
\newblock {$C\sp *$}-algebras generated by isometries and {W}iener-{H}opf
  operators.
\newblock {\em J. Operator Theory}, 27(1):17--52, 1992.

\bibitem{Popescu1998_Wold}
G.~Popescu.
\newblock Noncommutative {W}old decompositions for semigroups of isometries.
\newblock {\em Indiana Univ. Math. J.}, 47(1):277--296, 1998.

\bibitem{Popovici2004}
D.~Popovici.
\newblock A {W}old-type decomposition for commuting isometric pairs.
\newblock {\em Proc. Amer. Math. Soc.}, 132(8):2303--2314, 2004.

\bibitem{Raeburn2005}
I.~Raeburn.
\newblock {\em Graph algebras}, volume 103 of {\em CBMS Regional Conference
  Series in Mathematics}.
\newblock Published for the Conference Board of the Mathematical Sciences,
  Washington, DC; by the American Mathematical Society, Providence, RI, 2005.

\bibitem{Sarkar2014}
J.~Sarkar.
\newblock Wold decomposition for doubly commuting isometries.
\newblock {\em Linear Algebra Appl.}, 445:289--301, 2014.

\bibitem{Sehnem2022}
C.~F. Sehnem.
\newblock {${\rm C}^*$}-envelopes of tensor algebras of product systems.
\newblock {\em J. Funct. Anal.}, 283(12):Paper No. 109707, 31, 2022.

\bibitem{SZ2008}
A.~Skalski and J.~Zacharias.
\newblock Wold decomposition for representations of product systems of
  {$C^*$}-correspondences.
\newblock {\em Internat. J. Math.}, 19(4):455--479, 2008.

\bibitem{Slocinski1980}
M.~S\l~oci\'{n}ski.
\newblock On the {W}old-type decomposition of a pair of commuting isometries.
\newblock {\em Ann. Polon. Math.}, 37(3):255--262, 1980.

\bibitem{Spielberg2012}
J.~Spielberg.
\newblock {$C^\ast$}-algebras for categories of paths associated to the
  {B}aumslag-{S}olitar groups.
\newblock {\em J. Lond. Math. Soc. (2)}, 86(3):728--754, 2012.

\bibitem{Spielberg2020}
J.~Spielberg.
\newblock Groupoids and {$C^*$}-algebras for left cancellative small
  categories.
\newblock {\em Indiana Univ. Math. J.}, 69(5):1579--1626, 2020.

\bibitem{Starling2015}
C.~Starling.
\newblock Boundary quotients of {$\rm C^*$}-algebras of right {LCM} semigroups.
\newblock {\em J. Funct. Anal.}, 268(11):3326--3356, 2015.

\bibitem{Suciu1968}
I.~Suciu.
\newblock On the semi-groups of isometries.
\newblock {\em Studia Math.}, 30:101--110, 1968.

\bibitem{Szep1950}
J.~Sz\'{e}p.
\newblock On the structure of groups which can be represented as the product of
  two subgroups.
\newblock {\em Acta Sci. Math. (Szeged)}, 12:57--61, 1950.

\bibitem{Wiart2016}
J.~Wiart.
\newblock Dilating covariant representations of a semigroup dynamical system
  arising from number theory.
\newblock {\em Integral Equations Operator Theory}, 84(2):217--233, 2016.

\bibitem{Zappa1940}
G.~Zappa.
\newblock Sulla costruzione dei gruppi prodotto di due dati sottogruppi
  permutabili tra loro.
\newblock In {\em Atti {S}econdo {C}ongresso {U}n. {M}at. {I}tal., {B}ologna,
  1940}, pages 119--125. Ed. Cremonese, Rome, 1942.

\end{thebibliography}
\end{document}